\documentclass[12pt]{amsart}

\usepackage{accents}
\usepackage{appendix}
\usepackage{amsfonts}
\usepackage{amsmath}
\usepackage{amssymb}	
\usepackage{amsthm,bm}
\usepackage{array,booktabs,multirow}
\usepackage{braket}
\usepackage{centernot}
\usepackage{cite}
\usepackage{comment}
\usepackage{dsfont}
\usepackage{mathalfa}
\usepackage[shortlabels]{enumitem}
\usepackage{etoolbox}
\usepackage{float}
\usepackage[hang, flushmargin]{footmisc}
\usepackage{latexsym}
\usepackage{lipsum}
\usepackage{needspace}
\usepackage{tikz}
\usepackage{hyperref}
\usetikzlibrary{matrix,arrows}
\usepackage{tikz-cd}
\usepackage{diagbox}
\usepackage{adjustbox}
\usepackage{multicol}

\makeatletter
\def\namedlabel#1#2{\begingroup
   \def\@currentlabel{#2}%
   \label{#1}\endgroup
}
\makeatother

\theoremstyle{plain}
\newtheorem{thm}{Theorem}[section]

\newtheorem{lem}[thm]{Lemma}

\theoremstyle{definition}

\newtheorem{defn}[thm]{Definition}

\theoremstyle{remark}

\setlist[enumerate,1]{leftmargin=2em}

\def\A{\mathcal A}

\def\C{\mathbb C}

\newcommand{\imi}{\mathbf{i}}

\def\N{\mathbb N}

\def\Z{\mathbb Z}
\def\even{{\mathrm{(e)}}}
\def\odd{{\mathrm{(o)}}}

\def\U{U(\mathfrak{so}_3)}

\def\Uq{U_q^{\prime }(\mathfrak{so}_3)}
\def\e{\varepsilon}

\newcommand{\floor}[1]{\left\lfloor #1 \right\rfloor}

\title[Askey--Wilson algebras and $\mathfrak{so}_{3}$]
{The Askey--Wilson algebras, the Lie algebra $\mathfrak{so}_{3}$, and their fermionic realizations}

\author{Hau-Wen Huang}

\address[H.-W. Huang]{
Department of Mathematics\\
National Central University\\
Taoyuan, Taiwan}
\email{hauwenh@math.ncu.edu.tw}

\begin{document}

\begin{abstract}
This paper establishes a comprehensive algebraic framework linking the Lie algebra $\mathfrak{so}_{3}$ to the Askey--Wilson algebras. First, we provide a manifestly symmetric reformulation of the algebra homomorphism from the universal Racah algebra $\Re$ to $U(\mathfrak{sl}_2)$ by exploiting a Lie algebra isomorphism between $\mathfrak{sl}_{2}$ and $\mathfrak{so}_{3}$. This perspective facilitates a natural extension to the quantum setting, where we construct an explicit algebra homomorphism from the universal Askey--Wilson algebra $\triangle_{q^4}$ to the nonstandard quantum algebra $U_{q}^{\prime}(\mathfrak{so}_{3})$. By viewing the finite-dimensional irreducible $U_{q}^{\prime}(\mathfrak{so}_{3})$-modules of classical type as $\triangle_{q^4}$-modules, we demonstrate that the decomposition patterns perfectly parallel the branching rules of $U(\mathfrak{so}_3)$ over $\Re$. Furthermore, we extend this correspondence to the fermionic setting by establishing algebra isomorphisms between the skew group rings over $U(\mathfrak{so}_3)$ and $U_q'(\mathfrak{so}_3)$ and their associated anticommutator spin algebras. Collectively, these results provide a unified correspondence that bridges the gap between integrable algebraic structures, quantum groups, and their fermionic analogues.
\end{abstract}

\maketitle



\allowdisplaybreaks

\section{Introduction}\label{sec:intro}

The Askey--Wilson algebras provide the foundational algebraic  structures for characterizing the bispectral property of orthogonal polynomials within the classical and $q$-Askey schemes. Since their inception \cite{hidden_sym,LP&dual}, these algebras have emerged as a unifying thread across diverse mathematical landscapes. This influence 
originates from the fine-grained analysis of representation theory 
\cite{Askeyscheme,Vidunas:2007,Huang:2015,Huang:2021,Huang:2026,SH:2019-1,Huang:BImodule} 
and the realizations within Lie algebras and the symmetries of superintegrable models \cite{K&sl2,gvz2013,SH:2020,R&LD2014,BI2014-2,Huang:CG,CG2013,Vinet2019}, which provide the algebraic backbone for the Askey scheme.
This impact naturally permeates the discrete geometry of $P$- and $Q$-polynomial association schemes \cite{TerAlgebraI,TerAlgebraII,TerAlgebraIII,BannaiIto1984,Huang:CG&Johnson,Huang:CG&Grassmann,Hamming:2021,hypercube2002,halved:2023,halved:2024,odd:2026}, 
the fundamental connections to double affine Hecke algebras \cite{koo07,koo08,daha&AW,gvz2013,nilDAHA&qHahn:2018,daha&Z3, Huang:R<BImodules,R&BI2015,Huang:AW&DAHAmodule,DAHA&OP:2019}, 
the symmetries of the Yang--Baxter equation and universal $R$-matrices \cite{Huang:RW,RW:2020,qBI2018,qAW&BI:2019,qAW&Skein:2025}, and the algebraic constructions of spin models \cite{cur2007-1,cur2007-2,cur1999,nomura2021,nomura2025}. 
A prominent member of this family, the Racah algebras, dates back to the 1965 study of angular momentum coupling in quantum mechanics \cite{Levy1965}, predating the formal introduction of Askey--Wilson algebras. They govern the bispectral property of Wilson and Racah polynomials at the apex of the classical Askey scheme. 
To unify these diverse manifestations, the universal Racah algebra $\Re$ was formulated as a comprehensive framework.

From the perspective of representation theory, the universal enveloping algebra $U(\mathfrak{sl}_{2})$ provides a natural environment for realizing the Askey--Wilson algebras. 
This is exemplified by the action of the Chevalley generators $E,F,H$ as raising, lowering, and weight operators on the Boolean lattice of subsets \cite{sl2&poset1980, sl2&poset1982}. 
This viewpoint has illuminated a series of connections between $U(\mathfrak{sl}_{2})$ and the Askey--Wilson algebras in relation to hypercubes, halved cubes, and the odd graphs  as well as Hamming, Johnson, and  Grassmann schemes \cite{halved:2024, hypercube2002,odd:2026,Huang:CG&Hamming,
Huang:CG&Johnson,Huang:CG&Grassmann}. 
For instance, an algebra homomorphism from $\Re$ into $U(\mathfrak{sl}_2)$ (cf. Theorem \ref{thm:R->U(sl2)}) was recently established \cite{halved:2024}, drawing motivation from the geometric link between Johnson graphs and hypercubes. However, while this construction provides structural insights, its expression in terms of $E,F,H$ involves cumbersome linear combinations that obscure the inherent cyclic symmetry of the Racah relations.

The first part of this paper demonstrates that a more transparent and canonical reformulation arises by transitioning to the Lie algebra $\mathfrak{so}_{3}$. By exploiting the Cartesian generators $I_1,I_2,I_3$, we reveal that the generators $A,B,C$ of $\Re$ map to remarkably symmetric quadratic forms, as established in Theorem \ref{thm:R->U(so3)}. This perspective not only clarifies the classical correspondence but also provides the essential foundation for extending the theory to the quantum setting. The primary goal of this paper is to establish a canonical $q$-analogue of this classical relationship. 
To this end, we consider the universal Askey--Wilson algebra $\triangle_q$, which governs the bispectral property of Askey--Wilson and Racah polynomials at the peak of the $q$-Askey scheme  \cite{uaw2011}.
In the classical limit $q\to 1$, both the Askey--Wilson and $q$-Racah polynomials degenerate into Wilson and Racah polynomials, reflecting the transition from the $q$-Askey scheme to the classical Askey scheme. Correspondingly, the algebra $\triangle_{q}$ emerges as the canonical $q$-analogue of the universal Racah algebra $\Re$.  
To preserve the structural integrity and cyclic symmetry found in the realization of $\Re$ within $U(\mathfrak{so}_{3})$, we adopt the nonstandard quantum algebra $U_{q}^{\prime }(\mathfrak{so}_{3})$ \cite{so31991} as the natural quantum counterpart. Unlike the standard Drinfel'd--Jimbo quantum groups defined via Cartan subalgebras and root vectors, $U_{q}^{\prime }(\mathfrak{so}_{3})$ is constructed by deforming the Cartesian generators of $\mathfrak{so}_{3}$. As emphasized by Gavrilik and Klimyk, this nonstandard approach is essential for preserving the classical branching structure, thereby enabling the construction of representations through Gel'fand--Tsetlin-type formulas---a feature lost in the standard Drinfel'd--Jimbo-type deformation of $\mathfrak{so}_{n}$.

By focusing on the algebra $\triangle_{q^4}$, we construct an explicit algebra homomorphism into $\Uq$, presented in Theorem \ref{thm:UAW->Uq(so3)}, wherein the generators $A,B,C$ of $\triangle_{q^4}$ are mapped to symmetric quadratic expressions in the Cartesian generators. Notably, this mapping is conceptually distinct from previously known $\mathfrak{sl}_2$-type homomorphisms \cite{uaw&equit2011}, but rather an independent realization rooted in the intrinsic $q$-rotational symmetry of $\Uq$. 
This $q$-analogue is further validated from a representation-theoretic perspective. For $q$ not a root of unity, it is well-established that finite-dimensional irreducible $\Uq$-modules are classified into classical and non-classical types \cite{qso3:1998, qso3:1999}. The classical type modules are of particular significance as they recover the finite-dimensional irreducible $U(\mathfrak{so}_{3})$-modules in the limit $q\to 1$. 
It was shown in \cite{halved:2024} that irreducible $U(\mathfrak{so}_{3})$-modules decompose into $\Re$-modules according to explicit branching rules (Theorems \ref{thm:dec_Lne} and \ref{thm:dec_Lno}). 
We demonstrate that the decomposition patterns of classical-type $\Uq$-modules over $\triangle _{q^{4}}$ (Theorems \ref{thm:decLn_even} and \ref{thm:decLn_odd}) perfectly parallel these classical counterparts. This remarkable one-to-one correspondence in the branching rules underscores the representation-theoretic rigidity of our construction, providing compelling evidence for its status as the canonical $q$-analogue of the relationship between $\Re$ and $\U$.

Finally, we extend this symmetric correspondence to the fermionic realm by considering the anticommutator spin algebra  $\mathcal{A}$ \cite{ACSA2003}.
By constructing an ingenious isomorphism between the skew group rings of $\mathbb{Z}/2\mathbb{Z}$ over $U(\mathfrak{so}_{3})$ and $\mathcal{A}$, we reveal that the universal Racah algebra $\Re$ admits a natural realization within $\mathcal{A}$ (Theorems \ref{thm:AZ/2Z->UZ/2Z} and \ref{thm:R->A}). 
This construction is further extended to the quantum setting, where a parallel isomorphism connects the skew group rings of $\mathbb{Z}/2\mathbb{Z}$ over $\Uq$ and its associated quantum fermionic analogue (Theorems \ref{thm:AqZ/2Z->UqZ/2Z} and \ref{thm:UAW->Aq}). 
Collectively, these results intertwine integrable algebraic structures, quantum groups, and fermionic systems into a unified and coherent architecture, transcending previous cumbersome formulations.

The paper is organized as follows. Section \ref{s:R->U(so3)} reformulates the algebra homomorphism $\Re \to U(\mathfrak{sl}_{2})$ into a manifestly symmetric mapping $\Re\to \U$. Section \ref{s:UAW->Uq(so3)} introduces the canonical $q$-analogue $\triangle_{q^4} \to \Uq$ and provides a rigorous verification. After reviewing the branching rules for $\U$-modules as $\Re$-modules in Section \ref{s:dec_classical}, we establish the decomposition rules for $\Uq$-modules viewed as $\triangle_{q^{4}}$-modules in Section \ref{s:dec_qanalog}. Section \ref{s:anticommutator} utilizes skew group ring isomorphisms to realize both $\Re$ and its quantum counterpart within anticommutator spin algebras. Finally, we discuss potential higher-rank generalizations in Section \ref{s:remark}.

\section{A manifestly symmetric relationship between $\Re$ and $\mathfrak{so}_{3}$}\label{s:R->U(so3)}

Throughout this paper, all algebras are assumed to be unital and associative, and all algebra homomorphisms are unital. For any two elements $x,y$ of an algebra, we write 
$$
[x,y]=xy-yx,
\qquad 
\{x,y\}=xy+yx.
$$
Let $\N$ be the set of all nonnegative integers and $\Z$ be the additive group of integers.
Let $\imi$ denote the imaginary unit in the complex field $\C$.

In this section, we recast the representation of $\Re$ into a more symmetric form by exploiting a Lie algebra isomorphism between $\mathfrak{sl}_2$ and $\mathfrak{so}_3$. We begin by recalling the formal definition of the universal Racah algebra.

\begin{defn}
[\!\!\cite{Levy1965,zhedanov1988,SH:2017-1}]
\label{defn:URA}
The {\it universal Racah algebra} $\Re$ is an algebra over $\C$ defined by generators and relations. The generators are $A,B,C,\Delta$ and the relations state that 
\begin{align}
\label{URA-1}
[A,B]=[B,C]=[C,A]=2\Delta
\end{align}
and the elements
\begin{align}
&[A,\Delta]+ AC- BA,
\label{URA-2}
\\
&[B,\Delta]+BA-CB,
\label{URA-3}
\\
&[C,\Delta]+CB-AC
\label{URA-4}
\end{align}
are central in $\Re$. In view of \eqref{URA-1}, the algebra $\Re$ is generated by $A,B,C$. 
We write $\alpha,\beta,\gamma$ for the central elements \eqref{URA-2}--\eqref{URA-4} of $\Re$, respectively.
\end{defn}

The Lie algebra $\mathfrak{sl}_2$ consists of all complex $2\times 2$ matrices with trace zero. The Chevalley basis for $\mathfrak{sl}_2$ is 
$$
E=\begin{pmatrix}
0 &1
\\
0 &0
\end{pmatrix},
\qquad 
F=\begin{pmatrix}
0 &0
\\
1 &0
\end{pmatrix},
\qquad 
H=\begin{pmatrix}
1 &0
\\
0 &-1
\end{pmatrix}.
$$
These elements satisfy the relations
\begin{align}
[H,E]&=2E,
\label{sl2-1}
\\
[H,F]&=-2F,
\label{sl2-2}
\\
[E,F]&=H.
\label{sl2-3}
\end{align}
The universal enveloping algebra $U(\mathfrak{sl}_{2})$ of $\mathfrak{sl}_{2}$ is the algebra over $\mathbb{C}$ generated by $E,F,H$ subject to \eqref{sl2-1}--\eqref{sl2-3}. 
Motivated by the discrete geometry of Johnson graphs, the following realization of $\Re$ in $U(\mathfrak{sl}_{2})$ was obtained in terms of $E,F,H$:

\begin{thm}
[Theorem 1.3, \cite{halved:2024}]
\label{thm:R->U(sl2)}
There exists a unique algebra homomorphism $\Re\to U(\mathfrak{sl}_2)$ that sends 
\begin{eqnarray*}
A &\mapsto & \frac{(E+F-2)(E+F+2)}{16},
\\
B &\mapsto & \frac{(H-2)(H+2)}{16},
\\
C &\mapsto & \frac{(\imi E-\imi F-2)(\imi E-\imi F+2)}{16}.
\end{eqnarray*}
Moreover, the homomorphism maps each of $\alpha,\beta,\gamma$ to zero.
\end{thm}

While Theorem \ref{thm:R->U(sl2)} provides a concrete realization of $\Re$ within $U(\mathfrak{sl}_{2})$, the assignment of $A,B,C$ appears somewhat asymmetric, obscuring the intrinsic $\Z/3\Z$-symmetry of the relations \eqref{URA-2}--\eqref{URA-4}. To resolve this, we shift our focus to the Lie algebra $\mathfrak{so}_{3}$, whose Cartesian basis is naturally suited to reflect this cyclic invariance.

Recall that the Lie algebra $\mathfrak{so}_3$ consists of all complex $3\times 3$ skew-symmetric matrices. A Cartesian basis for $\mathfrak{so}_3$ is 
$$
I_1=\begin{pmatrix}
0 &1 &0
\\
-1 &0 &0
\\
0 &0 &0
\end{pmatrix},
\qquad 
I_2=\begin{pmatrix}
0 &0 &0
\\
0 &0 &1
\\
0 &-1 &0
\end{pmatrix},
\qquad 
I_3=\begin{pmatrix}
0 &0 &1
\\
0 &0 &0 
\\
-1 &0 &0
\end{pmatrix}.
$$
This basis satisfies the relations
\begin{align}
[I_1,I_2]&=I_3,
\label{so3-1}
\\
[I_2,I_3]&=I_1,
\label{so3-2}
\\
[I_3,I_1]&=I_2.
\label{so3-3}
\end{align}
By \eqref{sl2-1}--\eqref{sl2-3} and \eqref{so3-1}--\eqref{so3-3}, there exists a Lie algebra isomorphism $\mathfrak{so}_3\to \mathfrak{sl}_2$ given by
\begin{eqnarray*}
I_1 &\mapsto & \frac{\imi(E+F)}{2},
\\
I_2 &\mapsto &  \frac{\imi H}{2},
\\
I_3 &\mapsto & \frac{E-F}{2}.
\end{eqnarray*}
This Lie algebra isomorphism extends to an algebra isomorphism $U(\mathfrak{so}_3)\to U(\mathfrak{sl}_2)$. Through this identification, the realization of  $\Re$ in Theorem \ref{thm:R->U(sl2)} is revealed to admit the following remarkably symmetric form:

\begin{thm}
\label{thm:R->U(so3)}
There exists a unique algebra homomorphism $\Re\to U(\mathfrak{so}_3)$ that
sends
\begin{eqnarray*}
A &\mapsto & 
-\frac{I_1^2+1}{4},
\\
B &\mapsto & 
-\frac{I_2^2+1}{4},
\\
C &\mapsto & 
-\frac{I_3^2+1}{4}.
\end{eqnarray*}
Moreover, the homomorphism maps each of $\alpha,\beta,\gamma$ to zero.
\end{thm}

\section{The canonical $q$-analogue of the realization of $\Re$ within $\U$}\label{s:UAW->Uq(so3)}

Throughout, 
let $\C^\times$ be the multiplicative group of all nonzero complex numbers, and fix $q\in \C^\times $ that is not a root of unity. For any two elements $x,y$ in an algebra and any nonzero scalar $\nu$, we adopt the following $\nu$-commutator and $\nu $-anticommutator:
$$
[x,y]_\nu=\nu x y-\nu^{-1} yx,
\qquad 
\{x,y\}_\nu=\nu x y+\nu^{-1} y x.
$$
To underscore the structural correspondence between the classical and quantum cases, we will maintain a notational system for the $q$-analogues that coincides with the one used in the classical setting.

Just as $\Re$ governs the bispectral property of the Racah polynomials at the top of the classical Askey scheme, the algebra $\triangle_{q}$ plays a corresponding role for the Askey--Wilson and $q$-Racah polynomials at the peak of the $q$-Askey scheme. 
Thus, the universal Askey--Wilson algebra $\triangle_q$ stands as the canonical $q$-analogue of $\Re$, defined as follows:

\begin{defn}
[\!\!{\cite[Definition 1.2]{uaw2011}}]
\label{defn:UAW}
The {\it universal Askey--Wilson algebra} $\triangle_q$ is the algebra over $\C$ defined by the generators $A,B,C$ and the relations stating that 
\begin{align}\label{UAW}
A+\frac{[B,C]_q}{q^2-q^{-2}},
\qquad
B+\frac{[C,A]_q}{q^2-q^{-2}},
\qquad
C+\frac{[A,B]_q}{q^2-q^{-2}}
\end{align}
are central in $\triangle_q$. Let $\alpha,\beta,\gamma$ denote these central elements \eqref{UAW} multiplied by $q+q^{-1}$.
\end{defn}

To preserve the structural integrity and cyclic symmetry established in Theorem \ref{thm:R->U(so3)}, the nonstandard quantum algebra $U_{q}^{\prime }(\mathfrak{so}_{3})$ emerges as the natural  quantum analogue of $U(\mathfrak{so}_{3})$.

\begin{defn}
[\!\!{\cite[Section 2]{so31991}}]
\label{defn:Uq}
The {\it nonstandard quantum algebra} $\Uq$ of $U(\mathfrak{so}_3)$ is the algebra over $\C$ generated by $I_1,I_2,I_3$ subject to the relations 
\begin{align}
[I_1, I_2]_q
&=I_3,
\label{Uq-1}
\\
[I_2, I_3]_q
&=I_1,
\label{Uq-2}
\\
[I_3, I_1]_q
&=I_2.
\label{Uq-3}
\end{align}
\end{defn}

The element 
\begin{gather}
\label{Lambda}
\Lambda=q^2 I_1^2+q^{-2} I_2^2+q^2 I_3^2
-q(q^2-q^{-2})I_1I_2I_3
\end{gather}
lies in the center of $\Uq$ and is known as the {\it Casimir element} of $\Uq$.

To preserve the cyclic symmetry of the quadratic realization, we consider the universal Askey--Wilson algebra at the parameter $q^{4}$. This choice aligns the relations \eqref{UAW} with the Cartesian $q$-deformations \eqref{Uq-1}--\eqref{Uq-3}, yielding the following canonical $q$-analogue of Theorem \ref{thm:R->U(so3)}. For simplicity, we retain the symbols $A,B,C$ and $\alpha ,\beta ,\gamma$ for the generators and central elements of $\triangle_{q^4}$, respectively.

\begin{thm}
\label{thm:UAW->Uq(so3)}
There exists a unique algebra homomorphism $\triangle_{q^4}\to \Uq$ that sends 
\begin{eqnarray}
A &\mapsto & 2-(q^2-q^{-2})^2 I_1^2,
\label{UAW->Uq(so3):A}
\\
B &\mapsto & 2-(q^2-q^{-2})^2 I_2^2,
\label{UAW->Uq(so3):B}
\\
C &\mapsto & 2-(q^2-q^{-2})^2 I_3^2.
\label{UAW->Uq(so3):C}
\end{eqnarray}
Moreover, the homomorphism maps each of $\alpha,\beta,\gamma$ to 
\begin{gather}
\label{UAW->Uq(so3):abc}
2(q^2+q^{-2})^2-(q^2-q^{-2})(q^4-q^{-4}) \Lambda.
\end{gather}
\end{thm}

The remainder of this section is devoted to the proof of Theorem \ref{thm:UAW->Uq(so3)}.

\begin{lem}
\label{lem:I1I2I3}
The following equations hold in $\Uq$:
\begin{enumerate}
\item $I_2 I_1 I_3=q^2 I_1 I_2 I_3-q I_3^2$.

\item $I_1 I_3 I_2=q^2 I_1 I_2 I_3-q I_1^2$.

\item $I_2 I_3 I_1=I_1 I_2 I_3+q^{-1} I_2^2-q^{-1} I_3^2$.

\item $I_3 I_1 I_2=I_1 I_2 I_3+q^{-1} I_2^2-q^{-1} I_1^2$.
\end{enumerate}
\end{lem} 
\begin{proof}
(i): Rearranging relation \eqref{Uq-1} as 
$$
I_{2}I_{1}=q^2 I_{1}I_{2}-q I_{3}
$$ and multiplying both sides by $I_{3}$ from the right yields (i).

(ii): Similarly, rearranging relation \eqref{Uq-2} as 
$$
I_{3}I_{2}=q^2 I_{2}I_{3}-q I_{1}
$$ 
and multiplying both sides by $I_{1}$ from the left yields (ii).

For (iii) and (iv), we rewrite relation \eqref{Uq-3} as
\begin{gather} \label{I3I1}
I_3 I_1 = q^{-1} I_2 + q^{-2} I_1 I_3.
\end{gather}

(iii): Left-multiplying \eqref{I3I1} by $I_{2}$ gives 
$$
I_{2}I_{3}I_{1}=q^{-1}I_{2}^{2}+q^{-2} I_{2}I_{1}I_{3}.
$$ Substituting the expression for $I_{2}I_{1}I_{3}$ from (i) into this equation yields (iii).

(iv): Right-multiplying \eqref{I3I1} by $I_{2}$ gives 
$$
I_{3}I_{1}I_{2}=q^{-1}I_{2}^{2}+q^{-2} I_{1}I_{3}I_{2}.
$$ 
Substituting the expression for $I_{1}I_{3}I_{2}$ from (ii) into this equation yields (iv).
\end{proof}

\begin{lem}
\label{lem:[I1^2,I2^2]_q^2}
The following equations hold in $\Uq$:
\begin{enumerate}
\item $q^2 I_1 I_2 I_1 I_2-q^{-2} I_2 I_1 I_2 I_1=2 q I_1 I_2 I_3-I_1^2+I_2^2-I_3^2$.

\item $[I_1^2, I_2^2]_{q^4}=
q(q^2+q^{-2})^2 I_1 I_2 I_3-(q^2+q^{-2})(q^2 I_1^2-q^{-2}I_2^2+q^{-2} I_3^2)$.
\end{enumerate}
\end{lem}
\begin{proof}
(i): We expand the left-hand side of (i) by adding and subtracting $I_{2}I_{1}^{2}I_{2}$: 
$$
q^2 I_1 I_2 I_1 I_2 - q^{-2} I_2 I_1 I_2 I_1
= q [I_1,I_2]_q I_1 I_2
+  q^{-1} I_2 I_1 [I_1,I_2]_q.
$$
By \eqref{Uq-1} this expression simplifies to
$$
q I_3 I_1 I_2 + q^{-1} I_2 I_1 I_3.
$$
Applying Lemma \ref{lem:I1I2I3}(i) and (iv) to these cubic terms yields the right-hand side of (i).

(ii): Multiplying \eqref{Uq-1} from the left and right by $I_1$ and $I_2$ respectively (and vice versa for the second case), we have
\begin{align*}
I_1^2 I_2^2 &= 
q^{-2} I_1 I_2 I_1 I_2+ q^{-1} I_1 I_3 I_2,
\\
I_2^2 I_1^2 &=
q^2 I_2 I_1 I_2 I_1-q I_2 I_3 I_1.
\end{align*}
Therefore the left-hand side of (ii) can be written as
$$
q^2 I_1 I_2 I_1 I_2-q^{-2} I_2 I_1 I_2 I_1+q^3 I_1 I_3 I_2+q^{-3} I_2 I_3 I_1.
$$
Substituting Lemma \ref{lem:[I1^2,I2^2]_q^2}(i) and applying Lemma \ref{lem:I1I2I3}(ii) and (iii), we obtain the right-hand side of (ii).
\end{proof}

By Definition \ref{defn:Uq},
 the cyclic symmetry of the relations \eqref{Uq-1}--\eqref{Uq-3} induces a unique $\Z/3\Z$-action on $\Uq$ as a group of algebra automorphisms such that the generator $1+3\Z$ maps 
\begin{eqnarray*}
(I_1,I_2,I_3)
&\mapsto &
(I_2,I_3,I_1).
\end{eqnarray*}

\begin{lem}
\label{lem:Uq&Z/3Z}
The Casimir element $\Lambda$ of $\Uq$ is invariant under the $\Z/3\Z$-action.
\end{lem}
\begin{proof}
Since $1+3\Z$ generates $\Z/3\Z$ it suffices to show that $(1+3\Z)(\Lambda)=\Lambda$. 
Applying $1+3\Z$ to both sides of \eqref{Lambda} yields 
$$
(1+3\Z)(\Lambda)=q^2 I_1^2+q^2 I_2^2+q^{-2} I_3^2
-q(q^2-q^{-2})I_2 I_3 I_1.
$$
A direct calculation shows 
$$
(1+3\Z)(\Lambda)-\Lambda=(q^2-q^{-2})(I_2^2-I_3^2)
-q(q^2-q^{-2})(I_2 I_3 I_1-I_1 I_2 I_3).
$$
According to Lemma \ref{lem:I1I2I3}(iii) the difference $I_2 I_3 I_1-I_1 I_2 I_3=q^{-1}(I_2^2-I_3^2)$.
Substituting this into the above identity confirms that $(1+3\Z)(\Lambda)=\Lambda$.
The lemma follows.
\end{proof}

\begin{proof}[Proof of Theorem \ref{thm:UAW->Uq(so3)}] 
Throughout this proof, let $A,B,C$ denote the right-hand sides of \eqref{UAW->Uq(so3):A}--\eqref{UAW->Uq(so3):C}, respectively.
To establish Theorem \ref{thm:UAW->Uq(so3)}, since $\Lambda$ is central in $\Uq$, it suffices to show that the elements
\begin{gather}
\label{e:UAW->Uq(so3)}
A+\frac{[B,C]_{q^4}}{q^8-q^{-8}},
\qquad 
B+\frac{[C,A]_{q^4}}{q^8-q^{-8}},
\qquad 
C+\frac{[A,B]_{q^4}}{q^8-q^{-8}}
\end{gather}
are all equal to  \eqref{UAW->Uq(so3):abc} divided by $q^4+q^{-4}$.

A direct expansion of the $q^4$-commutator yields
$$
[A,B]_{q^4}=
(q^2-q^{-2})^4[I_1^2,I_2^2]_{q^4}
-2(q^2-q^{-2})^2(q^4-q^{-4})(I_1^2+I_2^2)
+4(q^4-q^{-4}).
$$
Applying Lemma \ref{lem:[I1^2,I2^2]_q^2}(ii) shows that the third element in \eqref{e:UAW->Uq(so3)} is 
\begin{gather*}
\frac{(q^2-q^{-2})(q^4-q^{-4})}{q^4+q^{-4}} \left( q(q^2-q^{-2}) I_1 I_2 I_3-q^2 I_1^2-q^{-2} I_2^2-q^2 I_3^2 +2\frac{q^2+q^{-2}}{(q^2-q^{-2})^2} \right).
\end{gather*}
In view of \eqref{Lambda}, the expression in the parenthesis simplifies to
$$
2\frac{q^2+q^{-2}}{(q^2-q^{-2})^2}-\Lambda.
$$
Consequently, we obtain 
$$
C+\frac{[A,B]_{q^4}}{q^8-q^{-8}}
=
\frac{2(q^2+q^{-2})^{2}-(q^2-q^{-2})(q^4-q^{-4})\Lambda }{q^4+q^{-4}},
$$
which is precisely \eqref{UAW->Uq(so3):abc} divided by $q^4+q^{-4}$.

By the $\Z/3\Z$-invariance of $\Lambda$ established in Lemma \ref{lem:Uq&Z/3Z} and the cyclic symmetry of the assignments \eqref{UAW->Uq(so3):A}--\eqref{UAW->Uq(so3):C}, the remaining identities in \eqref{e:UAW->Uq(so3)} follow immediately by applying the automorphisms $1+3\Z$ and $2+3\Z$. The result follows.
\end{proof}

\section{Branching rules of finite-dimensional $U(\mathfrak{so}_3)$-modules over $\Re$}
\label{s:dec_classical}

In this section, we recall the branching rules for finite-dimensional irreducible $U(\mathfrak{so}_3)$-modules viewed as $\Re$-modules. 
We begin by providing the necessary background on finite-dimensional irreducible $\Re$-modules.

Let $n\in \N$ and $a,b,c\in \C$ be given. By \cite[Proposition 2.4]{SH:2019-1}, there exists a unique $(n+1)$-dimensional $\Re$-module $V_n(a,b,c)$ up to isomorphism satisfying the following conditions:
\begin{enumerate}
\item[$\bullet$] There exists a basis $\{v_i\}_{i=0}^n$ for $V_n(a,b,c)$ such that 
\begin{align*}
A v_i &= \theta_i v_i + v_{i+1} \qquad (0\leq i\leq n),
\\
B v_i &= \theta_i^* v_i+\varphi_i v_{i-1} \qquad (0\leq i\leq n),
\end{align*}
where $v_{-1}$ and $v_{n+1}$ are interpreted as the zero vector of $V_n(a,b,c)$ and 
\begin{align*}
  \theta_i & =
  \left(a+\frac{n}{2}-i\right)
  \left(a+\frac{n}{2}-i+1\right)
  \qquad
  (0\leq i\leq n),
 \\
  \theta_i^* & =
  \left(b+\frac{n}{2}-i\right)
  \left(b+\frac{n}{2}-i+1\right)
  \qquad
  (0\leq i\leq n),
  \\
  \varphi_i &=i(i-n-1)
  \left(a+b+c+\frac{n}{2}-i+2\right)
  \left(a+b-c+\frac{n}{2}-i+1\right)
  \qquad
  (1\leq i\leq n).
\end{align*}

\item[$\bullet$] The central elements $\alpha,\beta,\gamma$ act on $V_n(a,b,c)$ as scalar multiplication by
\begin{align*}
&(c-b)(c+b+1)
\left(a-\frac{n}{2}\right)
\left(a+\frac{n}{2}+1\right),
\\
&(a-c)(a+c+1)
\left(b-\frac{n}{2}\right)
\left(b+\frac{n}{2}+1\right),
\\
&(b-a)(b+a+1)
\left(c-\frac{n}{2}\right)
\left(c+\frac{n}{2}+1\right),
\end{align*}
respectively.
\end{enumerate}

By \cite[Theorem 4.5]{SH:2019-1} the $\Re$-module $V_n(a,b,c)$ is irreducible if and only if 
$$
a+ b+ c+ 1,
-a+ b+ c,
a-b+ c,
a+ b-c
\not\in 
\left\{
\frac{n}{2}-i 
\;\middle|\;
i=1,2,\ldots, n
\right\}.
$$
A classification result \cite[Theorem 6.3]{SH:2019-1} asserts that any finite-dimensional irreducible $\Re$-module $V$ is isomorphic to $V_n(a,b,c)$ for some $n\in \N$ and $a,b,c\in \C$, where $\dim V=n+1$ and the parameters $a,b,c$ satisfy the aforementioned irreducibility condition. 
In practice, the parameters $a,b,c$ can be identified by referencing \cite[Corollary 6.5]{SH:2019-1}. Specifically, an $(n+1)$-dimensional irreducible $\Re$-module $V$ is isomorphic to $V_{n}(a,b,c)$ if and only if the traces of $A,B,C$ on $V$ are equal to $n+1$ times
\begin{align*}
  a(a+1)+\frac{n(n+2)}{12},
  \qquad
  b(b+1)+\frac{n(n+2)}{12},
  \qquad
  c(c+1)+\frac{n(n+2)}{12},
\end{align*}
respectively.

Having established the classification of finite-dimensional irreducible $\Re$-modules, we now turn our attention to the representation theory of $\U$. 
For each $n\in \N$, there exists a unique $(n+1)$-dimensional irreducible $\U$-module $L_n$ up to isomorphism. 
This module admits a basis $\{u_{i}\}_{i=0}^{n}$ such that the Cartesian generators act as follows:
\begin{align*}
I_1 u_i&=\imi(b_{n-i} u_{i-1}+ b_i u_{i+1})
\qquad (0\leq i\leq n),
\\
I _2 u_i&=\imi \theta_i u_i 
\qquad (0\leq i\leq n),
\\
I_3 u_i&=b_{n-i} u_{i-1}-b_i u_{i+1}
\qquad (0\leq i\leq n),
\end{align*}
where 
\begin{align*}
\theta_i&=\frac{n-2i}{2} \qquad (0\leq i\leq n),
\\
b_i &=\frac{i+1}{2} \qquad (0\leq i\leq n).
\end{align*}
By the complete reducibility of finite-dimensional $\U$-modules, it suffices to study the decomposition of the irreducible $\U$-module $L_n$ into $\Re$-modules. The $\Re$-module $L_n$ is explicitly given as follows.

\begin{lem}
\label{lem:RonLn}
For any $n\in \N$ the actions of $A,B,C$ on the $\Re$-module $L_n$ are as follows:
\begin{align*}
A u_i &=
\bar{b}_{n-i} u_{i-2}
+ \bar{a}_i u_i 
+\bar{b}_i u_{i+2} 
\qquad (0\leq i\leq n),
\\
B u_i &=
\bar{\theta}_i u_i
\qquad (0\leq i\leq n),
\\
C u_i &=
-\bar{b}_{n-i} u_{i-2}
+ \bar{a}_i u_i 
- \bar{b}_i u_{i+2}
\qquad (0\leq i\leq n),
\end{align*}
where $u_{-2}, u_{-1}, u_{n+1}, u_{n+2}$ are interpreted as the zero vector of $L_n$ and 
\begin{align*}
\bar{\theta}_i &= \frac{(n-2i+2)(n-2i-2)}{16}
\qquad (0\leq i\leq n),
\\
\bar{a}_i&=\frac{i(n-i+1)+(i+1)(n-i)}{16}-\frac{1}{4}
\qquad (0\leq i\leq n),
\\
\bar{b}_i&=\frac{(i+1)(i+2)}{16}
\qquad (0\leq i\leq n).
\end{align*}
\end{lem}
\begin{proof}
The formulas are obtained by applying the algebra homomorphism from Theorem \ref{thm:R->U(so3)} to the $U(\mathfrak{so}_{3})$-module $L_n$. 
\end{proof}

We decompose the basis $\{u_i\}_{i=0}^{n}$ according to the parity of the indices. Let 
\begin{align*}
u_i^\even &=u_{2i} \qquad (0\leq i\leq \floor{\tfrac{n}{2}}),
\\
u_i^\odd &=u_{2i+1} \qquad (0\leq i\leq \floor{\tfrac{n-1}{2}}).
\end{align*}
Let $L_n^\even$ and $L_n^\odd$ denote the subspaces of $L_n$ spanned by $\{u_i^\even\}_{i=0}^{\floor{\frac{n}{2}}}$ and 
$\{u_i^\odd\}_{i=0}^{\floor{\frac{n-1}{2}}}$, respectively. 
By construction, the vector space $L_n$ admits the  decomposition:
$$
L_n=L_n^\even\oplus L_n^\odd.
$$
Note that $L_0^\odd$ is the zero space.
It follows from Lemma \ref{lem:RonLn} that $L_n^\even$ and $L_n^\odd$ are $\Re$-submodules of $L_n$. Consequently, $L_n$ is the direct sum of $L_n^\even$ and $L_n^\odd$ as $\Re$-modules. Based on this decomposition, the branching rules for $L_{n}$ are summarized as follows.

\begin{thm}
[\!\!{\cite[Theorem 7.3]{halved:2024}}]
\label{thm:dec_Lne}
For any $n\in \N$ the following statements hold:
\begin{enumerate}
\item If $n$ is odd, then $L_n^\even$ is isomorphic to the irreducible $\Re$-module $V_{\frac{n-1}{2}}(-\frac{1}{4},-\frac{1}{4},-\frac{1}{4})$.

\item If $n=0$, then $L_n^\even$ is isomorphic to the irreducible $\Re$-module $V_0(-\frac{1}{2},-\frac{1}{2},-\frac{1}{2})$.

\item If $n\not=0$ and $n\equiv 0\pmod{4}$, then $L_n^\even$ is isomorphic to a direct sum of the irreducible $\Re$-modules $V_{\frac{n}{4}-1}(\frac{n-4}{8},\frac{n}{8},\frac{n-4}{8})$ and
$V_{\frac{n}{4}}(\frac{n-4}{8},\frac{n-4}{8},\frac{n-4}{8})$.

\item If $n\equiv 2\pmod{4}$, then $L_n^\even$ is isomorphic to a direct sum of the irreducible $\Re$-modules $V_{\frac{n-2}{4}}(\frac{n-6}{8},\frac{n-2}{8},\frac{n-2}{8})$ and
$V_{\frac{n-2}{4}}(\frac{n-2}{8},\frac{n-2}{8},\frac{n-6}{8})$.
\end{enumerate}
\end{thm}

\begin{thm}
[\!\!{\cite[Theorem 7.6]{halved:2024}}]
\label{thm:dec_Lno}
For any integer $n\geq 1$ the following statements hold:
\begin{enumerate}
\item If $n$ is odd, then $L_n^\odd$ is isomorphic to the irreducible $\Re$-module $
V_{\frac{n-1}{2}}(
-\frac{1}{4},-\frac{1}{4},-\frac{1}{4})$.

\item If $n=2$, then $L_n^\odd$ is isomorphic to the irreducible $\Re$-module $V_0(0,-\frac{1}{2},0)$.

\item If $n\not=2$ and $n\equiv2\pmod{4}$, then $L_n^\odd$ is isomorphic to a direct sum of the irreducible $\Re$-modules $V_{\frac{n-6}{4}}(\frac{n-2}{8},\frac{n-2}{8},\frac{n-2}{8})$ and
$V_{\frac{n-2}{4}}(\frac{n-2}{8},\frac{n-6}{8},\frac{n-2}{8})$.

\item If $n\equiv0\pmod{4}$, then $L_n^\odd$ is isomorphic to a direct sum of the irreducible $\Re$-modules 
$V_{\frac{n}{4}-1}(\frac{n-4}{8},\frac{n-4}{8},\frac{n}{8})$ 
and 
$V_{\frac{n}{4}-1}(\frac{n}{8},\frac{n-4}{8},\frac{n-4}{8})$.
\end{enumerate}
\end{thm}

These classical branching rules provide a template for the quantum case. In the next section, we will demonstrate that the $q$-deformed correspondence established in Theorem \ref{thm:UAW->Uq(so3)} preserves these decomposition patterns with remarkable rigidity.

\section{Branching rules of finite-dimensional $\Uq$-modules over $\triangle_{q^4}$}\label{s:dec_qanalog}

In this section, we establish the branching rules for finite-dimensional irreducible $\Uq$-modules viewed as $\triangle_{q^4}$-modules. 
Following the logical structure of the classical case,  we begin by providing the necessary background on finite-dimensional irreducible $\triangle_{q}$-modules.

Let $n\in \N$ and $a,b,c\in \C^\times$ be given. By \cite[Section 4.1]{Huang:2015}, there exists a unique $(n+1)$-dimensional $\triangle_{q}$-module $V_n(a,b,c)$ up to isomorphism satisfying the following conditions:
\begin{enumerate}
\item[$\bullet$] There exists a basis $\{v_i\}_{i=0}^n$ for $V_n(a,b,c)$ such that 
\begin{align*}
A v_i &= \theta_i v_i + v_{i+1} \qquad (0\leq i\leq n),
\\
B v_i &= \theta_i^* v_i+\varphi_i v_{i-1} \qquad (0\leq i\leq n),
\end{align*}
where $v_{-1}$ and $v_{n+1}$ are interpreted as the zero vector of $V_n(a,b,c)$ and  
\begin{align*}
  \theta_i & =
  a q^{2i-n}+a^{-1} q^{n-2i}
  \qquad
  (0\leq i\leq n),
 \\
  \theta_i^* & =
  b q^{2i-n}+b^{-1} q^{n-2i}
  \qquad
  (0\leq i\leq n),
  \\
  \varphi_i &=a^{-1} b^{-1} q^{n+1}
  (q^{i}-q^{-i})
  (q^{i-n-1}-q^{n-i+1})
  \\
  &\qquad \times \; (q^{-i}-a b c q^{i-n-1})
  (q^{-i}-a b c^{-1} q^{i-n-1})
  \qquad
  (1\leq i\leq n).
\end{align*}

\item[$\bullet$] The central elements $\alpha,\beta,\gamma$ act on $V_n(a,b,c)$ as scalar multiplication by
\begin{align*}
&(b+b^{-1})(c+c^{-1})+(a+a^{-1})(q^{n+1}+q^{-n-1}),
\\
&(c+c^{-1})(a+a^{-1})+(b+b^{-1})(q^{n+1}+q^{-n-1}),
\\
&(a+a^{-1})(b+b^{-1})+(c+c^{-1})(q^{n+1}+q^{-n-1}),
\end{align*}
respectively.
\end{enumerate}

By \cite[Theorem 4.4]{Huang:2015} the $\triangle_q$-module $V_n(a,b,c)$ is irreducible if and only if 
$$
abc, 
a^{-1}bc, 
ab^{-1}c, 
abc^{-1}
\notin 
\left\{
q^{2i-n-1}
\;\middle|\;
i=1,2,\dots,n
\right\}.
$$
A classification result \cite[Theorem 4.7]{Huang:2015} asserts that any finite-dimensional irreducible $\triangle_{q}$-module $V$ is isomorphic to $V_n(a,b,c)$ for some $n\in \N$ and $a,b,c\in \C^\times$, where $\dim V=n+1$ and the parameters $a,b,c$ satisfy the aforementioned irreducibility condition. 
For the effective determination of $a,b,c$ in practice, the following criterion can be employed, which naturally extends the classical method discussed in Section \ref{s:dec_classical}.

\begin{lem}
[\!\!{\cite[Corollary 4.10]{Huang:2015}}]
\label{thm:Vn(abc)_criterion}
Let $n\in \N$. Suppose that $V$ is an $(n+1)$-dimensional irreducible $\triangle_q$-module. For any $a,b,c\in \C^\times$, the $\triangle_q$-module $V_n(a,b,c)$ is isomorphic to $V$ if and only if the traces of $A,B,C$ on $V$ are equal to 
\begin{gather*}
(a+a^{-1}) \frac{q^{n+1}-q^{-n-1}}{q-q^{-1}},
\\
(b+b^{-1}) \frac{q^{n+1}-q^{-n-1}}{q-q^{-1}},
\\
(c+c^{-1}) \frac{q^{n+1}-q^{-n-1}}{q-q^{-1}},
\end{gather*}
respectively.
\end{lem}

To maintain consistency with the $q$-analogue realization in Theorem \ref{thm:UAW->Uq(so3)}, from this point forward, the notation $V_{n}(a,b,c)$ will refer to the corresponding $(n+1)$-dimensional $\triangle_{q^4}$-module. Under this convention, the parameter $q$ in the preceding formulas is replaced by $q^{4}$. 

Having established the classification of finite-dimensional irreducible $\triangle_q$-modules,  we now address the representation theory of $\Uq$. 
 Given that every finite-dimensional $\Uq$-module is completely reducible \cite{qso3:1999}, the investigation focuses on the decomposition of finite-dimensional irreducible $\Uq$-modules. Following the classification in \cite{qso3:1998, Hav2001}, these irreducible modules are categorized into two families: the classical and non-classical types. Accordingly, the remainder of this section is divided into two subsections to treat these cases separately.

\subsection{Decomposition of classical $\Uq$-modules over $\triangle_{q^4}$}

Let $n\in \N$ be given. According to \cite{qso3:1998,Hav2001} there exists a unique $(n+1)$-dimensional irreducible $\Uq$-module $L_n$ up to isomorphism satisfying the following conditions:
\begin{enumerate}
\item[$\bullet$] There exists a basis $\{u_i\}_{i=0}^n$ for $L_n$ satisfying 
\begin{align*}
(q^2-q^{-2})  I_1 u_i &=
\imi q (q^{n-2i} b_{n-i} u_{i-1}+ q^{2i-n} b_i u_{i+1})
\qquad (0\leq i\leq n),
\\
(q^2-q^{-2}) I_2 u_i &=
\imi \theta_i u_i
\qquad (0\leq i\leq n),
\\
(q^2-q^{-2})  I_3 u_i &=
b_{n-i} u_{i-1}-b_i u_{i+1}
\qquad (0\leq i\leq n),
\end{align*}
where $u_{-1}$ and $u_{n+1}$ are interpreted as the zero vector of $L_n$ and 
\begin{align*}
\theta_i &=q^{n-2i}-q^{2i-n}
\qquad (0\leq i\leq n),
\\
b_i&=
\frac{q^{2i+2}-q^{-2i-2}}{q^{n-2i}+q^{2i-n}}
\qquad (0\leq i\leq n).
\end{align*}

\item[$\bullet$] The Casimir element $\Lambda$ acts on $L_n$ as scalar multiplication by 
$$
-\frac{(q^n-q^{-n})(q^{n+2}-q^{-n-2})}{(q^2-q^{-2})^2}.
$$
\end{enumerate}
The $\Uq$-module $L_{n}$ is said to be of {\it classical} type, as it recovers the $(n+1)$-dimensional irreducible $\U$-module in the limit $q\to 1$.

\begin{lem}
\label{lem:UAWonLn}
For any $n\in \N$ the actions of $A,B,C$ on the $\triangle_{q^4}$-module $L_n$ are as follows:
\begin{align*}
A u_i &=
q^{2n-4i+4} \bar{b}_{n-i}
u_{i-2}
+ \bar{a}_i u_i 
+ q^{4i-2n+4} \bar{b}_i
u_{i+2}
\qquad (0\leq i\leq n),
\\
B u_i &=
\bar{\theta}_i u_i
\qquad (0\leq i\leq n),
\\
C u_i &=
 - \bar{b}_{n-i} u_{i-2}
 +\bar{a}_i u_i 
-\bar{b}_i u_{i+2} 
\qquad (0\leq i\leq n),
\end{align*}
where $u_{-2},u_{-1},u_{n+1},u_{n+2}$ are interpreted as the zero vector of $L_n$ and 
\begin{align*}
\bar{\theta}_i &= q^{2n-4i}+q^{4i-2n}
\qquad (0\leq i\leq n),
\\
\bar{a}_i&=\frac{(q^2+q^{-2})(q^n+q^{-n})(q^{n+2}+q^{-n-2})}
{(q^{n-2i-2}+q^{2i-n+2})(q^{n-2i+2}+q^{2i-n-2})}
\qquad (0\leq i\leq n),
\\
\bar{b}_i&=\frac{(q^{2i+2}-q^{-2i-2})(q^{2i+4}-q^{-2i-4})}
{(q^{n-2i}+q^{2i-n})(q^{n-2i-2}+q^{2i-n+2})}
\qquad (0\leq i\leq n).
\end{align*}
Moreover, the central elements $\alpha,\beta,\gamma$ act on $L_n$ as scalar multiplication by 
$$
(q^2+q^{-2})(q^n+q^{-n})(q^{n+2}+q^{-n-2}).
$$
\end{lem}
\begin{proof}
The formulas are obtained by applying the algebra homomorphism from Theorem \ref{thm:UAW->Uq(so3)} to the $\Uq$-module $L_n$. 
\end{proof}

To further elucidate the submodule structure, we partition the basis $\{u_{i}\}_{i=0}^{n}$ according to the parity of the indices. Let 
\begin{align*}
u_i^\even &=u_{2i} \qquad (0\leq i\leq \floor{\tfrac{n}{2}}),
\\
u_i^\odd &=u_{2i+1} \qquad (0\leq i\leq \floor{\tfrac{n-1}{2}}).
\end{align*}
Let $L_n^\even$ and $L_n^\odd$ denote the subspaces of $L_n$ spanned by $\{u_i^\even\}_{i=0}^{\floor{\frac{n}{2}}}$ and 
$\{u_i^\odd\}_{i=0}^{\floor{\frac{n-1}{2}}}$, respectively. 
By construction, the vector space $L_n$ admits the  decomposition:
$$
L_n=L_n^\even\oplus L_n^\odd.
$$
Note that $L_0^\odd$ is the zero space.
It follows from Lemma \ref{lem:UAWonLn} that $L_n^\even$ and $L_n^\odd$ are $\triangle_{q^4}$-submodules of $L_n$. Consequently, $L_n$ decomposes as the direct sum of $\triangle _{q^4}$-submodules $L_n^\even$ and $L_n^\odd$.

\begin{lem}
\label{lem:UAWonLn_even}
For any $n\in \N$ the actions of $A,B,C$ on $L_n^\even$ are as follows:
\begin{align*}
A u_i^\even &=
q^{2n-8i+4}  \bar{c}_i^{\even}
u_{i-1}^\even
+\bar{a}_i^\even u_i^\even
+ 
q^{8i-2n+4} \bar{b}_i^\even
u_{i+1}^\even
\qquad (0\leq i\leq \floor{\tfrac{n}{2}}),
\\
B u_i^\even &=\bar{\theta}_i^\even u_i^\even
\qquad (0\leq i\leq \floor{\tfrac{n}{2}}),
\\
C u_i^\even &=
- \bar{c}_i^{\even}  u_{i-1}^\even
+ \bar{a}_i^\even u_i^\even
- \bar{b}_i^\even u_{i+1}^\even
\qquad (0\leq i\leq \floor{\tfrac{n}{2}}),
\end{align*}
where $u_{-1}^\even=0$, $u_{\floor{\frac{n}{2}}+1}^\even=0$ and 
\begin{align}
\bar{\theta}_i^\even &=q^{2n-8i}+q^{8i-2n} 
\qquad (0\leq i\leq \floor{\tfrac{n}{2}}),
\label{thetae}
\\
\bar{a}_i^\even &=
\frac{(q^2+q^{-2})(q^n+q^{-n})(q^{n+2}+q^{-n-2})}
{(q^{n-4i-2}+q^{4i-n+2})(q^{n-4i+2}+q^{4i-n-2})}
\qquad (0\leq i\leq \floor{\tfrac{n}{2}}),
\label{ae}
\\
\bar{b}_i^\even &=\frac{(q^{4i+2}-q^{-4i-2})(q^{4i+4}-q^{-4i-4})}
{(q^{n-4i}+q^{4i-n})(q^{n-4i-2}+q^{4i-n+2})}
\qquad (0\leq i\leq \floor{\tfrac{n}{2}}),
\label{be}
\\
\bar{c}_i^{\even} &=\frac{(q^{2n-4i+2}-q^{4i-2n-2})(q^{2n-4i+4}-q^{4i-2n-4})}
{(q^{n-4i}+q^{4i-n})(q^{n-4i+2}+q^{4i-n-2})}
\qquad (0\leq i\leq \floor{\tfrac{n}{2}}).
\label{ce}
\end{align}
\end{lem}
\begin{proof}
The formulas follow from Lemma \ref{lem:UAWonLn} by restricting the indices to the even case. Specifically, the coefficients are identified as
\begin{gather*}
\bar{\theta}_i^\even = \bar{\theta}_{2i}, \quad 
\bar{a}_i^\even = \bar{a}_{2i}, \quad 
\bar{b}_i^\even = \bar{b}_{2i}, \quad 
\bar{c}_i^\even = \bar{b}_{n-2i}
\qquad 
(0\leq i\leq \floor{\tfrac{n}{2}}).
\end{gather*}
The evaluations \eqref{thetae}--\eqref{ce} are then obtained by direct substitution.
\end{proof}


\begin{lem}
\label{lem:parameterLn_even}
With the notation of Lemma \ref{lem:UAWonLn_even}, the following statements hold:
\begin{enumerate}
\item If $n$ is odd then the scalars $\{\bar{\theta}_i^\even \}_{i=0}^{\frac{n-1}{2}}$ are mutually distinct.

\item If $n$ is even then the scalars $\bar{\theta}_i^\even$ for all integers $i$ with $0\leq i\leq \frac{n}{4}$ are mutually distinct.

\item The scalars $\{\bar{b}_i^\even\}_{i=0}^{\floor{\frac{n}{2}}}$ and 
$\{\bar{c}_i^{\even}\}_{i=0}^{\floor{\frac{n}{2}}}$ are all nonzero.

\item If $n$ is even then
\begin{align*}
\bar{\theta}_i^\even =\bar{\theta}_{\frac{n}{2}-i}^\even,
\qquad
\bar{a}_i^\even = \bar{a}_{\frac{n}{2}-i}^\even, 
\qquad
\bar{b}_i^\even = \bar{c}_{\frac{n}{2}-i}^{\even}
\end{align*}
for all $i=0,1,\dots ,\frac{n}{2}$.

\item For each $i=0,1,\ldots,\floor{\frac{n}{2}}$ the scalar $\bar{a}_i^\even$ is equal to 
$$
\frac{(q^n+q^{-n})(q^{n+2}+q^{-n-2})}{2(q^2-q^{-2})}
\left(
\frac{q^{n-4i+2}-q^{4i-n-2}}{q^{n-4i+2}+q^{4i-n-2}}-
\frac{q^{n-4i-2}-q^{4i-n+2}}{q^{n-4i-2}+q^{4i-n+2}}
\right).
$$

\end{enumerate}
\end{lem}
\begin{proof}
(i), (ii): Since $q$ is not a root of unity, it follows from \eqref{thetae} that for any distinct $i, j \in \{0, 1, \ldots, \floor{\frac{n}{2}}\}$, the equality $\bar{\theta}_{i}^\even=\bar{\theta }_{j}^\even$ holds if and only if  $n=2(i+j)$. Statements (i) and (ii) follow immediately from this observation.

(iii): Under the assumption that $q$ is not a root of unity, the numerators of \eqref{be} and \eqref{ce} never vanish.

(iv): These identities are verified by direct substitution using the  formulas \eqref{thetae}--\eqref{ce}.

(v): 
This follows by applying the identity 
$$
\frac{2(q^{a-b}-q^{b-a})}{(q^a+q^{-a})(q^b+q^{-b})}=\frac{q^a-q^{-a}}{q^a+q^{-a}}-\frac{q^b-q^{-b}}{q^b+q^{-b}}
$$
to \eqref{ae} with $a=n-4i+2$ and $b=n-4i-2$.
\end{proof}

\begin{lem}
\label{lem:traceLn_even}
For any $n\in \N$ the following statements hold:
\begin{enumerate}
\item The trace of $B$ on $L_n^\even$ is equal to 
\begin{align*}
\begin{cases}
\displaystyle \frac{q^{2n+2}-q^{-2n-2}}{q^2-q^{-2}} 
\qquad &\hbox{if $n$ is odd},
\\
\displaystyle \frac{2(q^{2n+4}-q^{-2n-4})}{q^4-q^{-4}}
\qquad &\hbox{if $n$ is even}.
\end{cases}
\end{align*}

\item The traces of $A$ and $C$ on $L_n^\even$ are equal to 
\begin{align*}
\begin{cases}
\displaystyle \frac{q^{2n+2}-q^{-2n-2}}{q^2-q^{-2}} 
\qquad &\hbox{if $n$ is odd},
\\
\displaystyle \frac{q^{2n+2}-q^{-2n-2}}{q^2-q^{-2}} +1
\qquad &\hbox{if $n$ is even}.
\end{cases}
\end{align*}
\end{enumerate}
\end{lem}
\begin{proof}
(i): The trace of $B$ on $L_n^\even$ is equal to 
$
\sum_{i=0}^{\floor{n/2}}
\bar{\theta}_i^\even$. 
By the summation formula for geometric series, this sum evaluates to 
\begin{align*}
\sum_{i=0}^{\floor{\frac{n}{2}}}q^{2n-8i}+\sum_{i=0}^{\floor{\frac{n}{2}}} q^{8i-2n}
= 
q^{2n}\frac{q^{-8(\floor{n/2}+1)}-1}{q^{-8}-1} 
+ q^{-2n}\frac{ q^{8(\floor{n/2}+1)}-1}{q^8-1}.
\end{align*}
Statement (i) is obtained by simplifying the above expression according to the parity of $n$.

(ii): The traces of $A$ and $C$ on $L_n^\even$ are both equal to 
$
\sum_{i=0}^{\floor{n/2}}
\bar{a}_i^\even$. 
According to Lemma \ref{lem:parameterLn_even}(v), this sum is equal to $\frac{(q^n+q^{-n})(q^{n+2}+q^{-n-2})}{2(q^2-q^{-2})}$ multiplied by the telescoping sum 
\begin{align*}
\sum_{i=0}^{\floor{\frac{n}{2}}}
\frac{q^{n-4i+2}-q^{4i-n-2}}{q^{n-4i+2}+q^{4i-n-2}}-
\frac{q^{n-4i-2}-q^{4i-n+2}}{q^{n-4i-2}+q^{4i-n+2}}
=
\begin{cases}
\displaystyle \frac{q^{n+2}-q^{-n-2}}{q^{n+2}+q^{-n-2}}+\frac{q^{n}-q^{-n}}{q^n+q^{-n}}
\qquad &\hbox{if $n$ is odd},
\\
\displaystyle \frac{2(q^{n+2}-q^{-n-2})}{q^{n+2}+q^{-n-2}}
\qquad &\hbox{if $n$ is even}.
\end{cases}
\end{align*}
Simplifying the resulting expressions yields the values stated in (ii).
\end{proof}

The branching rules for $L_n^\even$ are summarized in the following theorem. These results are perfectly parallel to the decomposition patterns in Theorem \ref{thm:dec_Lne}.

\begin{thm}
\label{thm:decLn_even}
For any $n\in \N$ the following statements hold:
\begin{enumerate}
\item If $n$ is odd, then the $\triangle_{q^4}$-module $L_n^\even$ is isomorphic to the irreducible $\triangle_{q^4}$-module 
$
V_{\frac{n-1}{2}}(q^2,q^2,q^2)$.

\item If $n = 0$, then $L_n^\even$ is isomorphic to the irreducible $\triangle_{q^4}$-module $V_0 (1,1,1)$.

\item If $n \not= 0$ and $n \equiv 0 \pmod {4}$, then $L_n^\even$ is isomorphic to a direct sum of the irreducible $\triangle_{q^4}$-modules $V_{\frac{n}{4}-1}(q^n,q^{n+4},q^n)$ and $V_{\frac{n}{4}}(q^n,q^n,q^n)$.

\item  If $n \equiv 2 \pmod {4}$, then $L_n^\even$ is isomorphic to a direct sum of the irreducible $\triangle_{q^4}$-modules $V_{\frac{n-2}{4}}(q^{n-2},q^{n+2},q^{n+2})$ and $V_{\frac{n-2}{4}}(q^{n+2},q^{n+2},q^{n-2})$.
\end{enumerate}
\end{thm}
\begin{proof}
Throughout this proof, the parameters $\{\bar{\theta}_i^\even\}_{i=0}^{\floor{n/2}}$, $\{\bar{a}_i^\even\}_{i=0}^{\floor{n/2}}$, $\{\bar{b}_i^\even\}_{i=0}^{\floor{n/2}}$, $\{\bar{c}_i^\even\}_{i=0}^{\floor{n/2}}$ are given in \eqref{thetae}--\eqref{ce}.

(i): Suppose that $n$ is odd. 
Recall the actions of $A,B,C$ on $L_n^\even$ from Lemma \ref{lem:UAWonLn_even}. With respect to the ordered basis $\{u_i^\even\}_{i=0}^{\frac{n-1}{2}}$, the matrix representing $B$ is diagonal, while those representing $A$ and $C$ are tridiagonal. 
The irreducibility of $L_n^\even$ then follows from Lemma \ref{lem:parameterLn_even}(i) and (iii).
By Lemma \ref{lem:traceLn_even} the trace of each of $A,B,C$ on $L_n^\even$ is equal to
$$
\frac{q^{2n+2}-q^{-2n-2}}{q^2-q^{-2}}=
(q^2+q^{-2})
\frac{q^{2n+2}-q^{-2n-2}}{q^4-q^{-4}}.
$$
According to Lemma \ref{thm:Vn(abc)_criterion}, the irreducible $\triangle_{q^4}$-module $L_n^\even$ is isomorphic to $V_{\frac{n-1}{2}}(q^2,q^2,q^2)$.

(ii): Suppose that $n=0$. Since $L_n^\even$ is one-dimensional, it is an irreducible $\triangle_{q^4}$-module. 
By Lemma \ref{lem:traceLn_even} the traces of $A,B,C$ on $L_n^\even$ are equal to $2$. According to Lemma \ref{thm:Vn(abc)_criterion}, the irreducible $\triangle_{q^4}$-module $L_n^\even$ is isomorphic to $V_0(1,1,1)$.

(iii): Suppose that $n\not=0$ and $n\equiv 0\pmod{4}$. Let 
\begin{align*}
v_i&=u_i^\even-u_{\frac{n}{2}-i}^\even \qquad (0\leq i\leq  \tfrac{n}{4}-1),
\\
w_i&=u_i^\even+u_{\frac{n}{2}-i}^\even \qquad (0\leq i\leq   \tfrac{n}{4}).
\end{align*}
Observe that 
\begin{align*}
u_i^\even&= \frac{v_i+w_i}{2} \qquad (0\leq i\leq  \tfrac{n}{4}-1),
\\
u_{\frac{n}{2}-i}^\even&=\frac{w_i-v_i}{2} \qquad (0\leq i\leq  \tfrac{n}{4}-1),
\\
u_{\frac{n}{4}}&=\frac{1}{2}w_\frac{n}{4}.
\end{align*}
Since the vectors $\{u_i^\even\}_{i=0}^{\frac{n}{2}}$ are a basis for $L_n^\even$, this implies that $\{v_i\}_{i=0}^{\frac{n}{4}-1}$ and $\{w_i\}_{i=0}^{\frac{n}{4}}$ are a basis for $L_n^\even$. Let $V$ and $W$ denote the subspaces of $L_n^\even$ spanned by $\{v_i\}_{i=0}^{\frac{n}{4}-1}$ and $\{w_i\}_{i=0}^{\frac{n}{4}}$, respectively. Then 
$$
L_n^\even=V\oplus W.
$$

Recall the actions of $A,B,C$ on $L_n^\even$ from Lemma \ref{lem:UAWonLn_even}.
In view of Lemma \ref{lem:parameterLn_even}(iv), 
a direct calculation shows that $V$ and $W$ are two $\triangle_{q^4}$-submodules of $L_n^\even$. The actions of $A,B,C$ on $V$ are as follows:
\begin{align*}
A v_i&=
q^{2n-8i+4} \bar{c}_i^{\even} v_{i-1}
+\bar{a}_i^\even v_i
+q^{8i-2n+4} \bar{b}_i^\even  v_{i+1}
\qquad   (0\leq i\leq  \tfrac{n}{4}-1),
\\
B v_i&=
\bar{\theta}_i^\even v_i 
\qquad  (0\leq i\leq  \tfrac{n}{4}-1),
\\
C v_i&=
-\bar{c}_i^{\even} v_{i-1}
+\bar{a}_i^\even v_i
-\bar{b}_i^\even v_{i+1}
\qquad   (0\leq i\leq  \tfrac{n}{4}-1),
\end{align*}
where $v_{-1}=0$ and $v_{\frac{n}{4}}=0$. 
With respect to the ordered basis $\{v_i\}_{i=0}^{\frac{n}{4}-1}$, the matrix representing $B$ is diagonal, while those representing $A$ and $C$ are tridiagonal. 
The irreducibility of $V$ then follows from Lemma \ref{lem:parameterLn_even}(ii) and (iii). 
The trace of $B$ on $V$ is equal to 
$\sum_{i=0}^{\frac{n}{4}-1}\bar{\theta}_i^\even$. 
By \eqref{thetae} and Lemmas \ref{lem:parameterLn_even}(iv) and \ref{lem:traceLn_even}(i), this sum evaluates to
$$
\frac{1}{2}\left(\sum\limits_{i=0}^{\frac{n}{2}} \bar{\theta}_i^\even
-\bar{\theta}_{\frac{n}{4}}^\even \right)
=(q^{n+4}+q^{-n-4})\frac{q^n-q^{-n}}{q^4-q^{-4}}.
$$
The traces of $A$ and $C$ on $V$ are equal to 
$\sum_{i=0}^{\frac{n}{4}-1}\bar{a}_i^\even$.  
By \eqref{ae} and Lemmas \ref{lem:parameterLn_even}(iv) and \ref{lem:traceLn_even}(ii), this sum is equal to 
$$
\frac{1}{2}\left(\sum\limits_{i=0}^{\frac{n}{2}} \bar{a}_i^\even 
-\bar{a}_{\frac{n}{4}}^\even\right)
=(q^n+q^{-n})\frac{q^{n}-q^{-n}}{q^4-q^{-4}}.
$$
According to Lemma \ref{thm:Vn(abc)_criterion}, the irreducible $\triangle_{q^4}$-module $V$ is isomorphic to $V_{\frac{n}{4}-1}(q^n,q^{n+4},q^n)$.

The actions of $A,B,C$ on $W$ are as follows:
\begin{align*}
A w_i&=
q^{2n-8i+4}  \bar{c}_i^{\even} w_{i-1}
+\bar{a}_i^\even w_i
+q^{8i-2n+4} \bar{b}_i^\even w_{i+1}
\qquad  (0\leq i\leq  \tfrac{n}{4}-1),
\\
A w_{\frac{n}{4}}&=
2 q^4 \bar{c}_{\frac{n}{4}}^\even 
w_{\frac{n}{4}-1}
+\bar{a}_{\frac{n}{4}}^\even w_{\frac{n}{4}},
\\
B w_i&=\bar{\theta}_i^\even w_i \qquad  (0\leq i\leq \tfrac{n}{4}),
\\
C w_i&=-\bar{c}_i^{\even} w_{i-1}+\bar{a}_i^\even w_i-\bar{b}_i^\even w_{i+1}
\qquad   (0\leq i\leq \tfrac{n}{4}-1),
\\
C w_{\frac{n}{4}}&=
-2\bar{c}_{\frac{n}{4}}^{\even}
w_{\frac{n}{4}-1}
+\bar{a}_{\frac{n}{4}}^\even w_{\frac{n}{4}},
\end{align*}
where $w_{-1}=0$. 
With respect to the ordered basis $\{w_i\}_{i=0}^{\frac{n}{4}}$, the matrix representing $B$ is diagonal, while those representing $A$ and $C$ are tridiagonal. 
The irreducibility of $W$ then follows from Lemma \ref{lem:parameterLn_even}(ii) and (iii). 
The trace of $B$ on $W$ is equal to 
$\sum_{i=0}^{\frac{n}{4}}\bar{\theta}_i^\even$. 
By \eqref{thetae} and Lemmas \ref{lem:parameterLn_even}(iv) and \ref{lem:traceLn_even}(i), this sum evaluates to 
$$
\frac{1}{2}\left(\sum\limits_{i=0}^{\frac{n}{2}} \bar{\theta}_i^\even 
+\bar{\theta}_{\frac{n}{4}}^\even\right)
=(q^{n}+q^{-n})\frac{q^{n+4}-q^{-n-4}}{q^4-q^{-4}}.
$$
The traces of $A$ and $C$ on $W$ are equal to 
$\sum_{i=0}^{\frac{n}{4}}\bar{a}_i^\even$.  
By \eqref{ae} and Lemmas \ref{lem:parameterLn_even}(iv) and \ref{lem:traceLn_even}(ii), this sum is equal to 
$$
\frac{1}{2}\left(\sum\limits_{i=0}^{\frac{n}{2}} \bar{a}_i^\even 
+\bar{a}_{\frac{n}{4}}^\even\right)
=(q^{n}+q^{-n})\frac{q^{n+4}-q^{-n-4}}{q^4-q^{-4}}.
$$
According to Lemma \ref{thm:Vn(abc)_criterion},  the irreducible $\triangle_{q^4}$-module $W$ is isomorphic to $V_{\frac{n}{4}}(q^n,q^n,q^n)$.

(iv): Suppose that $n\equiv 2\pmod{4}$. Let 
\begin{align*}
v_i&=u_i^\even-u_{\frac{n}{2}-i}^\even \qquad (0\leq i\leq  \tfrac{n-2}{4}),
\\
w_i&=u_i^\even+u_{\frac{n}{2}-i}^\even \qquad (0\leq i\leq   \tfrac{n-2}{4}).
\end{align*}
Observe that 
\begin{align*}
u_i^\even&= \frac{v_i+w_i}{2} \qquad (0\leq i\leq  \tfrac{n-2}{4}),
\\
u_{\frac{n}{2}-i}^\even&=\frac{w_i-v_i}{2} \qquad (0\leq i\leq  \tfrac{n-2}{4}).
\end{align*}
Since the vectors $\{u_i^\even\}_{i=0}^{\frac{n}{2}}$ are a basis for $L_n^\even$, this implies that $\{v_i\}_{i=0}^{\frac{n-2}{4}}$ and $\{w_i\}_{i=0}^{\frac{n-2}{4}}$ are a basis for $L_n^\even$. Let $V$ and $W$ denote the subspaces of $L_n^\even$ spanned by $\{v_i\}_{i=0}^{\frac{n-2}{4}}$ and $\{w_i\}_{i=0}^{\frac{n-2}{4}}$, respectively.  Then 
$$
L_{n}^\even=V\oplus W.
$$

Recall the actions of $A,B,C$ on $L_n^\even$ from Lemma \ref{lem:UAWonLn_even}.
In view of Lemma \ref{lem:parameterLn_even}(iv)
a direct calculation shows that $V$ and $W$ are $\triangle_{q^4}$-submodules of $L_n^\even$. The actions of $A,B,C$ on $V$ are as follows:
\begin{align*}
A v_i&=
q^{2n-8i+4} \bar{c}_i^{\even} v_{i-1}
+\bar{a}_i^\even v_i
+q^{8i-2n+4} \bar{b}_i^\even v_{i+1}
\qquad   (0\leq i\leq \tfrac{n-6}{4}),
\\
A v_{\frac{n-2}{4}}&=
q^8 \bar{c}_{\frac{n-2}{4}}^{\even} v_{\frac{n-6}{4}}+
(\bar{a}_{\frac{n-2}{4}}^\even-\bar{b}_{\frac{n-2}{4}}^\even) v_\frac{n-2}{4}, 
\\
B v_i&=\bar{\theta}_i^\even v_i \qquad  (0\leq i\leq  \tfrac{n-2}{4}),
\\
C v_i&=-\bar{c}_i^{\even} v_{i-1}+\bar{a}_i^\even v_i-\bar{b}_i^\even v_{i+1}
\qquad   (0\leq i\leq \tfrac{n-6}{4}),
\\
C v_{\frac{n-2}{4}}&=
-\bar{c}_{\frac{n-2}{4}}^{\even} v_{\frac{n-6}{4}}+
(\bar{a}_{\frac{n-2}{4}}^\even+ \bar{b}_{\frac{n-2}{4}}^\even) v_{\frac{n-2}{4}},
\end{align*}
where $v_{-1}=0$. 
With respect to the ordered basis $\{v_i\}_{i=0}^{\frac{n-2}{4}}$, the matrix representing $B$ is diagonal, while those representing $A$ and $C$ are tridiagonal. 
The irreducibility of $V$ then follows from Lemma \ref{lem:parameterLn_even}(ii) and (iii). 
The trace of $A$ on $V$ is equal to 
$\sum_{i=0}^{\frac{n-2}{4}}\bar{a}_i^\even-\bar{b}_{\frac{n-2}{4}}^\even$.  
By \eqref{be} and Lemmas \ref{lem:parameterLn_even}(iv) and \ref{lem:traceLn_even}(ii), the sum is equal to 
$$
\frac{1}{2}
\sum\limits_{i=0}^{\frac{n}{2}} \bar{a}_i^\even 
-\bar{b}_{\frac{n-2}{4}}^\even
=(q^{n-2}+q^{2-n})\frac{q^{n+2}-q^{-n-2}}{q^4-q^{-4}}.
$$
The trace of $B$ on $V$ is equal to 
$\sum_{i=0}^{\frac{n-2}{4}}\bar{\theta}_i^\even$. 
By Lemmas \ref{lem:parameterLn_even}(iv) and \ref{lem:traceLn_even}(i), this sum evaluates to 
$$
\frac{1}{2}
\sum\limits_{i=0}^{\frac{n}{2}}  \bar{\theta}_i^\even 
=(q^{n+2}+q^{-n-2})\frac{q^{n+2}-q^{-n-2}}{q^4-q^{-4}}.
$$
The trace of $C$ on $V$ is equal to 
$\sum_{i=0}^{\frac{n-2}{4}}\bar{a}_i^\even+\bar{b}_{\frac{n-2}{4}}^\even$.  
By \eqref{be} and Lemmas \ref{lem:parameterLn_even}(iv) and  \ref{lem:traceLn_even}(ii), this sum is equal to 
$$
\frac{1}{2}
\sum\limits_{i=0}^{\frac{n}{2}} \bar{a}_i^\even 
+\bar{b}_{\frac{n-2}{4}}^\even
=(q^{n+2}+q^{-n-2})\frac{q^{n+2}-q^{-n-2}}{q^4-q^{-4}}.
$$
According to Lemma \ref{thm:Vn(abc)_criterion}, the irreducible $\triangle_{q^4}$-module $V$ is isomorphic to $V_{\frac{n-2}{4}}(q^{n-2},q^{n+2},q^{n+2})$.

The actions of $A,B,C$ on $W$ are as follows:
\begin{align*}
A w_i&=
q^{2n-8i+4} \bar{c}_i^{\even} w_{i-1}+
\bar{a}_i^\even w_i+
q^{8i-2n+4}\bar{b}_i^\even w_{i+1}
\qquad  (0\leq i\leq \tfrac{n-6}{4}),
\\
A w_{\frac{n-2}{4}} &=
q^8 \bar{c}_{\frac{n-2}{4}}^{\even} w_{\frac{n-6}{4}}
+(\bar{a}_{\frac{n-2}{4}}^\even+\bar{b}_{\frac{n-2}{4}}^\even) w_{\frac{n-2}{4}},
\\
B w_i&=\bar{\theta}_i^\even w_i \qquad  (0\leq i\leq \tfrac{n-2}{4}),
\\
C w_i&=-\bar{c}_i^{\even} w_{i-1}+\bar{a}_i^\even w_i-\bar{b}_i^\even w_{i+1}
\qquad   (0\leq i\leq \tfrac{n-6}{4}),
\\
C w_{\frac{n-2}{4}} &=
-\bar{c}_{\frac{n-2}{4}}^{\even} w_{\frac{n-6}{4}}
+(\bar{a}_{\frac{n-2}{4}}^\even-\bar{b}_{\frac{n-2}{4}}^\even) w_{\frac{n-2}{4}},
\end{align*}
where $w_{-1}=0$. 
With respect to the ordered basis $\{w_i\}_{i=0}^{\frac{n-2}{4}}$, the matrix representing $B$ is diagonal, while those representing $A$ and $C$ are tridiagonal. 
The irreducibility of $W$ then follows from Lemma \ref{lem:parameterLn_even}(ii) and (iii). 
The trace of $A$ on $W$ is equal to 
$\sum_{i=0}^{\frac{n-2}{4}}\bar{a}_i^\even+\bar{b}_{\frac{n-2}{4}}^\even$.  
By \eqref{be} and Lemmas \ref{lem:parameterLn_even}(iv) and \ref{lem:traceLn_even}(ii), this sum is equal to 
$$
\frac{1}{2}
\sum\limits_{i=0}^{\frac{n}{2}} \bar{a}_i^\even 
+\bar{b}_{\frac{n-2}{4}}^\even
=(q^{n+2}+q^{-n-2})\frac{q^{n+2}-q^{-n-2}}{q^4-q^{-4}}.
$$
The trace of $B$ on $W$ is equal to 
$\sum_{i=0}^{\frac{n-2}{4}}\bar{\theta}_i^\even$. 
By Lemmas \ref{lem:parameterLn_even}(iv) and \ref{lem:traceLn_even}(i), this sum evaluates to 
$$
\frac{1}{2}
\sum\limits_{i=0}^{\frac{n}{2}} \bar{\theta}_i^\even 
=(q^{n+2}+q^{-n-2})\frac{q^{n+2}-q^{-n-2}}{q^4-q^{-4}}.
$$
The trace of $C$ on $W$ is equal to 
$\sum_{i=0}^{\frac{n-2}{4}}\bar{a}_i^\even-\bar{b}_{\frac{n-2}{4}}^\even$.  
By \eqref{be} and Lemmas \ref{lem:parameterLn_even}(iv) and \ref{lem:traceLn_even}(ii), it is equal to 
$$
\frac{1}{2}\sum\limits_{i=0}^{\frac{n}{2}} \bar{a}_i^\even 
-\bar{b}_{\frac{n-2}{4}}^\even
=(q^{n-2}+q^{2-n})\frac{q^{n+2}-q^{-n-2}}{q^4-q^{-4}}.
$$
By Lemma \ref{thm:Vn(abc)_criterion}, the irreducible $\triangle_{q^4}$-module $W$ is isomorphic to $V_{\frac{n-2}{4}}(q^{n+2},q^{n+2},q^{n-2})$.
\end{proof}

\begin{lem}
\label{lem:UAWonLn_odd}
For any integer $n\geq 1$ the actions of $A,B,C$ on $L_n^\odd$ are as follows:
\begin{align*}
A u_i^\odd &=
q^{2n-8i}  \bar{c}_i^{\odd} u_{i-1}^\odd
+\bar{a}_i^\odd u_i^\odd
+ q^{8i-2n+8} \bar{b}_i^\odd u_{i+1}^\odd
\qquad (0\leq i\leq \floor{\tfrac{n-1}{2}}),
\\
B u_i^\odd &=\bar{\theta}_i^\odd u_i^\odd
\qquad (0\leq i\leq \floor{\tfrac{n-1}{2}}),
\\
C u_i^\odd &=
- \bar{c}_i^{\odd}  u_{i-1}^\odd
+\bar{a}_i^\odd u_i^\odd
- \bar{b}_i^\odd u_{i+1}^\odd
\qquad (0\leq i\leq \floor{\tfrac{n-1}{2}}),
\end{align*}
where $u_{-1}^\odd$ and $u_{\floor{\frac{n+1}{2}}}^\odd$ are interpreted as the zero vector of  $L_n^\odd$ and 
\begin{align}
\bar{\theta}_i^\odd &=q^{2n-8i-4}+q^{8i-2n+4} 
\qquad (0\leq i\leq \floor{\tfrac{n-1}{2}}),
\label{thetao}
\\
\bar{a}_i^\odd &=
\frac{(q^2+q^{-2})(q^n+q^{-n})(q^{n+2}+q^{-n-2})}
{(q^{n-4i-4}+q^{4i-n+4})(q^{n-4i}+q^{4i-n})}
\qquad (0\leq i\leq \floor{\tfrac{n-1}{2}}),
\label{ao}
\\
\bar{b}_i^\odd &=\frac{(q^{4i+4}-q^{-4i-4})(q^{4i+6}-q^{-4i-6})}
{(q^{n-4i-2}+q^{4i-n+2})(q^{n-4i-4}+q^{4i-n+4})}
\qquad (0\leq i\leq \floor{\tfrac{n-1}{2}}),
\label{bo}
\\
\bar{c}_i^{\odd} &=\frac{(q^{2n-4i}-q^{4i-2n})(q^{2n-4i+2}-q^{4i-2n-2})}
{(q^{4i-n+2}+q^{n-4i-2})(q^{4i-n}+q^{n-4i})}
\qquad (0\leq i\leq \floor{\tfrac{n-1}{2}}).
\label{co}
\end{align}
\end{lem}
\begin{proof}
The formulas follow from Lemma \ref{lem:UAWonLn} by restricting the indices to the odd case. Specifically, the coefficients are identified as
\begin{gather*}
\bar{\theta}_i^\odd = \bar{\theta}_{2i+1}, \quad 
\bar{a}_i^\odd = \bar{a}_{2i+1}, \quad 
\bar{b}_i^\odd = \bar{b}_{2i+1}, \quad 
\bar{c}_i^\odd = \bar{b}_{n-2i-1}
\qquad 
(0\leq i\leq \floor{\tfrac{n-1}{2}}).
\end{gather*}
 The evaluations \eqref{thetao}--\eqref{co} are then obtained by direct substitution.
\end{proof}


\begin{lem}
\label{lem:parameterLn_odd}
With the notation of Lemma \ref{lem:UAWonLn_odd}, the following statements hold:
\begin{enumerate}
\item If $n$ is odd then the scalars $\{\bar{\theta}_i^\odd \}_{i=0}^{\frac{n-1}{2}}$ are mutually distinct.

\item If $n$ is even then the scalars $\bar{\theta}_i^\odd$ for all integers $i$ with $0\leq i\leq \frac{n}{4}-1$ are mutually distinct.

\item The scalars 
$\{\bar{b}_i^\odd\}_{i=0}^{\floor{\frac{n-1}{2}}}$ 
and 
$\{\bar{c}_i^\odd\}_{i=0}^{\floor{\frac{n-1}{2}}}$ 
are all nonzero.

\item If $n$ is even then
\begin{align*}
\bar{\theta}_i^\odd =\bar{\theta}_{\frac{n}{2}-i-1}^\odd,
\qquad
\bar{a}_i^\odd = \bar{a}_{\frac{n}{2}-i-1}^\odd, 
\qquad
\bar{b}_i^\odd = \bar{c}_{\frac{n}{2}-i-1}^{\odd}
\end{align*}
for all $i=0,1,\dots ,\frac{n}{2}-1$.

\item For each $i=0,1,\ldots,\floor{\frac{n-1}{2}}$ the scalar $\bar{a}_i^\odd$ is equal to 
$$
\frac{(q^n+q^{-n})(q^{n+2}+q^{-n-2})}{2(q^2-q^{-2})}
\left(
\frac{q^{n-4i}-q^{4i-n}}{q^{n-4i}+q^{4i-n}}-
\frac{q^{n-4i-4}-q^{4i-n+4}}{q^{n-4i-4}+q^{4i-n+4}}
\right).
$$

\end{enumerate}
\end{lem}
\begin{proof}
(i), (ii): Since $q$ is not a root of unity, it follows from \eqref{thetao} that for any distinct $i, j \in \{0, 1, \ldots, \floor{\frac{n-1}{2}}\}$, the equality $\bar{\theta }_{i}^\odd=\bar{\theta }_{j}^\odd$ holds if and only if  $n=2(i+j+1)$. Statements (i) and (ii) follow immediately from this observation.

(iii): Since $q$ is not a root of unity, the numerators of \eqref{bo} and \eqref{co} never vanish.

(iv): These identities are verified by direct substitution using the  formulas \eqref{thetao}--\eqref{co}.

(v): This follows by applying the identity
$$
\frac{2(q^{a-b}-q^{b-a})}{(q^a+q^{-a})(q^b+q^{-b})}=\frac{q^a-q^{-a}}{q^a+q^{-a}}-\frac{q^b-q^{-b}}{q^b+q^{-b}}
$$
to  \eqref{ao}  with $a=n-4i$ and $b=n-4i-4$.
\end{proof}

\begin{lem}
\label{lem:traceLn_odd}
For any integer $n\geq 1$ the following statements hold:
\begin{enumerate}
\item The trace of $B$ on $L_n^\odd$ is equal to 
\begin{align*}
\begin{cases}
\displaystyle \frac{q^{2n+2}-q^{-2n-2}}{q^2-q^{-2}} 
\qquad &\hbox{if $n$ is odd},
\\
\displaystyle 2\cdot \frac{q^{2n}-q^{-2n}}{q^4-q^{-4}}
\qquad &\hbox{if $n$ is even}.
\end{cases}
\end{align*}

\item The traces of $A$ and $C$ on $L_n^\odd$ are equal to 
\begin{align*}
\begin{cases}
\displaystyle \frac{q^{2n+2}-q^{-2n-2}}{q^2-q^{-2}} 
\qquad &\hbox{if $n$ is odd},
\\
\displaystyle \frac{(q^{n+2}+q^{-n-2})(q^{n}-q^{-n})}{q^2-q^{-2}}
\qquad &\hbox{if $n$ is even}.
\end{cases}
\end{align*}

\end{enumerate}
\end{lem}
\begin{proof}
(i): The trace of $B$ on $L_n^\odd$ is equal to 
$
\sum_{i=0}^{\floor{(n-1)/2}}
\bar{\theta}_i^\odd$. 
By the summation formula for geometric series, this sum evaluates to 
\begin{align*}
\sum_{i=0}^{\floor{\frac{n-1}{2}}}q^{2n-8i-4}+\sum_{i=0}^{\floor{\frac{n-1}{2}}} q^{8i-2n+4}
=
q^{2n-4}\frac{q^{-8\floor{(n+1)/2}}-1}{q^{-8}-1}
+q^{4-2n}\frac{q^{8\floor{(n+1)/2}}-1}{q^{8}-1}.
\end{align*}
Statement (i) is obtained by simplifying the above expression according to the parity of $n$.

(ii): The traces of $A$ and $C$ on $L_n^\odd$ are both equal to 
$
\sum_{i=0}^{\floor{(n-1)/2}}
\bar{a}_i^\odd$. 
According to Lemma \ref{lem:parameterLn_odd}(v), this sum is equal to $\frac{(q^n+q^{-n})(q^{n+2}+q^{-n-2})}{2(q^2-q^{-2})}$ multiplied by the telescoping sum 
\begin{align*}
\sum_{i=0}^{\floor{\frac{n-1}{2}}}
\frac{q^{n-4i}-q^{4i-n}}{q^{n-4i}+q^{4i-n}}-
\frac{q^{n-4i-4}-q^{4i-n+4}}{q^{n-4i-4}+q^{4i-n+4}}
=
\begin{cases}
\displaystyle \frac{q^{n}-q^{-n}}{q^{n}+q^{-n}}+\frac{q^{n+2}-q^{-n-2}}{q^{n+2}+q^{-n-2}}
\qquad &\hbox{if $n$ is odd},
\\
\displaystyle 2\cdot \frac{q^{n}-q^{-n}}{q^{n}+q^{-n}}
\qquad &\hbox{if $n$ is even}.
\end{cases}
\end{align*}
Simplifying the resulting expressions yields the values stated in (ii).
\end{proof}

The branching rules for $L_n^\odd$ are summarized in the following theorem. These results are perfectly parallel to the decomposition patterns in Theorem \ref{thm:dec_Lno}.

\begin{thm}
\label{thm:decLn_odd}
For any integer $n\geq 1$ the following statements hold:
\begin{enumerate}
\item If $n$ is odd, then the $\triangle_{q^4}$-module $L_n^\odd$ is isomorphic to the irreducible $\triangle_{q^4}$-module 
$
V_{\frac{n-1}{2}}(q^2,q^2,q^2)$.

\item If $n = 2$, then $L_n^\odd$ is isomorphic to the irreducible $\triangle_{q^4}$-module $V_0(q^4,1,q^4)$.

\item If $n \not= 2$ and $n \equiv 2 \pmod {4}$, then $L_n^\odd$ is isomorphic to a direct sum of the irreducible $\triangle_{q^4}$-modules 
$V_{\frac{n-6}{4}}(q^{n+2},q^{n+2},q^{n+2})$ and 
$V_{\frac{n-2}{4}}(q^{n+2},q^{n-2},q^{n+2})$.

\item  If $n \equiv 0 \pmod {4}$, then $L_n^\odd$ is isomorphic to a direct sum of the irreducible $\triangle_{q^4}$-modules $V_{\frac{n}{4}-1} (q^{n},q^{n},q^{n+4})$ and $V_{\frac{n}{4}-1}(q^{n+4},q^{n},q^{n})$.
\end{enumerate}
\end{thm}
\begin{proof}
Throughout this proof, the parameters $\{\bar{\theta}_i^\odd\}_{i=0}^{\floor{(n-1)/2}}$, $\{\bar{a}_i^\odd\}_{i=0}^{\floor{(n-1)/2}}$, $\{\bar{b}_i^\odd\}_{i=0}^{\floor{(n-1)/2}}$, $\{\bar{c}_i^\odd\}_{i=0}^{\floor{(n-1)/2}}$ are given in \eqref{thetao}--\eqref{co}.

(i): Suppose that $n$ is odd. 
Recall the actions of $A,B,C$ on $L_n^\odd$ from Lemma \ref{lem:UAWonLn_odd}.
With respect to the ordered basis $\{u_i^\odd\}_{i=0}^{\frac{n-1}{2}}$, the matrix representing $B$ is diagonal, while those representing $A$ and $C$ are tridiagonal. 
The irreducibility of $L_n^\odd$ then follows from Lemma \ref{lem:parameterLn_odd}(i) and (iii). 
By Lemma \ref{lem:traceLn_odd}, the trace of each of $A,B,C$ on $L_n^\odd$ is equal to
$$
\frac{q^{2n+2}-q^{-2n-2}}{q^2-q^{-2}}=
(q^2+q^{-2})\frac{q^{2n+2}-q^{-2n-2}}{q^4-q^{-4}}.
$$
According to Lemma \ref{thm:Vn(abc)_criterion},  the irreducible $\triangle_{q^4}$-module $L_n^\odd$ is isomorphic to $V_{\frac{n-1}{2}}(q^2,q^2,q^2)$.

(ii): Suppose that $n=2$. Since $L_n^\odd$ is one-dimensional, it is an irreducible $\triangle_{q^4}$-module. 
By Lemma \ref{lem:traceLn_odd}(i) the trace of $B$ on $L_n^\odd$ is equal to $2$. By Lemma \ref{lem:traceLn_odd}(ii) the traces of $A$ and $C$ on $L_n^\odd$ are equal to $q^4+q^{-4}$. 
According to Lemma \ref{thm:Vn(abc)_criterion}, the irreducible $\triangle_{q^4}$-module $L_n^\odd$ is isomorphic to $V_0(q^4,1,q^4)$.

(iii): Suppose that $n\not=2$ and $n\equiv 2\pmod{4}$. Let 
\begin{align*}
v_i&=u_i^\odd-u_{\frac{n}{2}-i-1}^\odd \qquad (0\leq i\leq  \tfrac{n-6}{4}),
\\
w_i&=u_i^\odd+u_{\frac{n}{2}-i-1}^\odd \qquad (0\leq i\leq \tfrac{n-2}{4}).
\end{align*}
Observe that 
\begin{align*}
u_i^\odd&= \frac{v_i+w_i}{2} \qquad (0\leq i\leq  \tfrac{n-6}{4}),
\\
u_{\frac{n}{2}-i-1}^\odd&=\frac{w_i-v_i}{2} \qquad (0\leq i\leq \tfrac{n-6}{4}),
\\
u_{\frac{n-2}{4}}&=\frac{1}{2}w_\frac{n-2}{4}.
\end{align*}
Since the vectors $\{u_i^\odd\}_{i=0}^{\frac{n}{2}-1}$ are a basis for $L_n^\odd$, this implies that $\{v_i\}_{i=0}^{\frac{n-6}{4}}$ and $\{w_i\}_{i=0}^{\frac{n-2}{4}}$ are a basis for $L_n^\odd$. Let $V$ and $W$ denote the subspaces of $L_n^\odd$ spanned by $\{v_i\}_{i=0}^{\frac{n-6}{4}}$ and $\{w_i\}_{i=0}^{\frac{n-2}{4}}$, respectively. Then 
$$
L_n^\odd=V\oplus W.
$$

Recall the actions of $A,B,C$ on $L_n^\odd$ from Lemma \ref{lem:UAWonLn_odd}.
In view of Lemma \ref{lem:parameterLn_odd}(iv), 
a direct calculation shows that $V$ and $W$ are $\triangle_{q^4}$-submodules of $L_n^\odd$. The actions of $A,B,C$ on $V$ are as follows:
\begin{align*}
A v_i&=
q^{2n-8i} \bar{c}_i^{\odd} v_{i-1}+
\bar{a}_i^\odd v_i+
q^{8i-2n+8} \bar{b}_i^\odd v_{i+1}
\qquad   (0\leq i\leq  \tfrac{n-6}{4}),
\\
B v_i&=
\bar{\theta}_i^\odd v_i 
\qquad  (0\leq i\leq  \tfrac{n-6}{4}),
\\
C v_i&=
-\bar{c}_i^{\odd} v_{i-1}+
\bar{a}_i^\odd v_i
-\bar{b}_i^\odd v_{i+1}
\qquad   (0\leq i\leq \tfrac{n-6}{4}),
\end{align*}
where $v_{-1}=0$ and $v_{\frac{n-2}{4}}=0$. 
With respect to the ordered basis $\{v_i\}_{i=0}^{\frac{n-6}{4}}$, the matrix representing $B$ is diagonal, while those representing $A$ and $C$ are tridiagonal. 
The irreducibility of $V$ then follows from Lemma \ref{lem:parameterLn_odd}(ii) and (iii). 
The trace of $B$ on $V$ is equal to 
$\sum_{i=0}^{(n-6)/4}\bar{\theta}_i^\odd$. 
By \eqref{thetao} and Lemmas \ref{lem:parameterLn_odd}(iv) and \ref{lem:traceLn_odd}(i), this sum evaluates to
$$
\frac{1}{2}\left(\sum\limits_{i=0}^{\frac{n}{2}-1} \bar{\theta}_i^\odd
-\bar{\theta}_{\frac{n-2}{4}}^\odd \right)
=(q^{n+2}+q^{-n-2})\frac{q^{n-2}-q^{2-n}}{q^4-q^{-4}}.
$$
The traces of $A$ and $C$ on $V$ are equal to 
$\sum_{i=0}^{(n-6)/4}\bar{a}_i^\odd$.  
By \eqref{ao} and Lemmas \ref{lem:parameterLn_odd}(iv) and \ref{lem:traceLn_odd}(ii), this sum is equal to 
$$
\frac{1}{2}\left(\sum\limits_{i=0}^{\frac{n}{2}-1} \bar{a}_i^\odd
-\bar{a}_{\frac{n-2}{4}}^\odd \right)
=(q^{n+2}+q^{-n-2})\frac{q^{n-2}-q^{2-n}}{q^4-q^{-4}}.
$$
According to Lemma \ref{thm:Vn(abc)_criterion}, the irreducible $\triangle_{q^4}$-module $V$ is isomorphic to $V_{\frac{n-6}{4}}(q^{n+2},q^{n+2},q^{n+2})$.

The actions of $A,B,C$ on $W$ are as follows:
\begin{align*}
A w_i&=
q^{2n-8i} \bar{c}_i^{\odd} w_{i-1}
+\bar{a}_i^\odd w_i
+q^{8i-2n+8} \bar{b}_i^\odd w_{i+1}
\qquad  (0\leq i\leq  \tfrac{n-6}{4}),
\\
A w_{\frac{n-2}{4}}&=
2q^4 \bar{c}_{\frac{n-2}{4}}^\odd w_{\frac{n-6}{4}}
+\bar{a}_{\frac{n-2}{4}}^\odd w_{\frac{n-2}{4}},
\\
B w_i&=\bar{\theta}_i^\odd w_i \qquad  (0\leq i\leq  \tfrac{n-2}{4}),
\\
C w_i&=-\bar{c}_i^{\odd} w_{i-1}+\bar{a}_i^\odd w_i-\bar{b}_i^\odd w_{i+1}
\qquad   (0\leq i\leq  \tfrac{n-6}{4}),
\\
C w_{\frac{n-2}{4}}&=
-2\bar{c}_{\frac{n-2}{4}}^\odd w_{\frac{n-6}{4}}
+\bar{a}_{\frac{n-2}{4}}^\odd w_{\frac{n-2}{4}}
\end{align*}
where $w_{-1}=0$. With respect to the ordered basis $\{w_i\}_{i=0}^{\frac{n-2}{4}}$, the matrix representing $B$ is diagonal, while those representing $A$ and $C$ are tridiagonal. 
The irreducibility of $W$ then follows from Lemma \ref{lem:parameterLn_odd}(ii) and (iii). 
The trace of $B$ on $W$ is equal to 
$\sum_{i=0}^{\frac{n-2}{4}}\bar{\theta}_i^\odd$. 
By \eqref{thetao} and Lemmas \ref{lem:parameterLn_odd}(iv) and \ref{lem:traceLn_odd}(i), this sum evaluates to 
$$
\frac{1}{2}\left(\sum\limits_{i=0}^{\frac{n}{2}-1} \bar{\theta}_i^\odd 
+\bar{\theta}_{\frac{n-2}{4}}^\odd\right)
=(q^{n-2}+q^{2-n})\frac{q^{n+2}-q^{2-n}}{q^4-q^{-4}}.
$$
The traces of $A$ and $C$ on $W$ are equal to 
$\sum_{i=0}^{\frac{n-2}{4}}\bar{a}_i^\odd$.  
By \eqref{ao} and Lemmas \ref{lem:parameterLn_odd}(iv) and \ref{lem:traceLn_even}(ii), this sum is equal to 
$$
\frac{1}{2}\left(\sum\limits_{i=0}^{\frac{n}{2}-1} \bar{a}_i^\odd 
+\bar{a}_{\frac{n}{4}}^\odd\right)
=(q^{n+2}+q^{-n-2})\frac{q^{n+2}-q^{-n-2}}{q^4-q^{-4}}.
$$
According to Lemma \ref{thm:Vn(abc)_criterion},  the irreducible $\triangle_{q^4}$-module $W$ is isomorphic to $V_{\frac{n-2}{4}}(q^{n+2},q^{n-2},q^{n+2})$.

(iv): Suppose that $n\equiv 0\pmod{4}$. Let 
\begin{align*}
v_i&=u_i^\odd-u_{\frac{n}{2}-i-1}^\odd \qquad (0\leq i\leq  \tfrac{n}{4}-1),
\\
w_i&=u_i^\odd+u_{\frac{n}{2}-i-1}^\odd \qquad (0\leq i\leq    \tfrac{n}{4}-1).
\end{align*}
Observe that 
\begin{align*}
u_i^\odd&= \frac{v_i+w_i}{2} \qquad (0\leq i\leq \tfrac{n}{4}-1),
\\
u_{\frac{n}{2}-i-1}^\odd&=\frac{w_i-v_i}{2} \qquad (0\leq i\leq \tfrac{n}{4}-1).
\end{align*}
Since the vectors $\{u_i^\odd\}_{i=0}^{\frac{n}{2}-1}$ are a basis for $L_n^\odd$, this implies that $\{v_i\}_{i=0}^{\frac{n}{4}-1}$ and $\{w_i\}_{i=0}^{\frac{n}{4}-1}$ are a basis for $L_n^\odd$. Let $V$ and $W$ denote the subspaces of $L_n^\odd$ spanned by $\{v_i\}_{i=0}^{\frac{n}{4}-1}$ and $\{w_i\}_{i=0}^{\frac{n}{4}-1}$, respectively.  Then 
$$
L_{n}^\odd=V\oplus W.
$$

Recall the actions of $A,B,C$ on $L_n^\odd$ from Lemma \ref{lem:UAWonLn_odd}.
In view of Lemma \ref{lem:parameterLn_odd}(iv)
a direct calculation shows that $V$ and $W$ are $\triangle_{q^4}$-submodules of $L_n^\odd$. The actions of $A,B,C$ on $V$ are as follows:
\begin{align*}
A v_i&=
q^{2n-8i} \bar{c}_i^{\odd} v_{i-1}
+\bar{a}_i^\odd v_i
+q^{8i-2n+8} \bar{b}_i^\odd v_{i+1}
\qquad   (0\leq i\leq  \tfrac{n}{4}-2),
\\
A v_{\frac{n}{4}-1}&=
q^8 \bar{c}_{\frac{n}{4}-1}^{\odd} v_{\frac{n}{4}-2}+
(\bar{a}_{\frac{n}{4}-1}^\odd-\bar{b}_{\frac{n}{4}-1}^\odd) v_{\frac{n}{4}-1}, 
\\
B v_i&=
\bar{\theta}_i^\odd v_i 
\qquad  (0\leq i\leq  \tfrac{n}{4}-1),
\\
C v_i&=
-\bar{c}_i^{\odd} v_{i-1}
+\bar{a}_i^\odd v_i
-\bar{b}_i^\odd v_{i+1}
\qquad   (0\leq i\leq  \tfrac{n}{4}-2),
\\
C v_{\frac{n}{4}-1}&=
-\bar{c}_{\frac{n}{4}-1}^{\odd} v_{\frac{n}{4}-2}+
(\bar{a}_{\frac{n}{4}-1}^\odd+ \bar{b}_{\frac{n}{4}-1}^\odd ) v_{\frac{n}{4}-1},
\end{align*}
where $v_{-1}=0$. 
With respect to the ordered basis $\{v_i\}_{i=0}^{\frac{n}{4}-1}$, the matrix representing $B$ is diagonal, while those representing $A$ and $C$ are tridiagonal. 
The irreducibility of $V$ then follows from Lemma \ref{lem:parameterLn_odd}(ii) and (iii). 
The trace of $A$ on $V$ is equal to 
$\sum_{i=0}^{\frac{n}{4}-1}\bar{a}_i^\odd-\bar{b}_{\frac{n}{4}-1}^\odd$.  
By \eqref{bo} and Lemmas \ref{lem:parameterLn_odd}(iv) and \ref{lem:traceLn_odd}(ii), the sum is equal to 
$$
\frac{1}{2}
\sum\limits_{i=0}^{\frac{n}{2}-1} \bar{a}_i^\odd 
-\bar{b}_{\frac{n}{4}-1}^\odd
=(q^{n}+q^{-n})\frac{q^{n}-q^{-n}}{q^4-q^{-4}}.
$$
The trace of $B$ on $V$ is equal to 
$\sum_{i=0}^{\frac{n}{4}-1}\bar{\theta}_i^\odd$. 
By Lemmas \ref{lem:parameterLn_odd}(iv) and \ref{lem:traceLn_odd}(i), this sum evaluates to 
$$
\frac{1}{2}
\sum\limits_{i=0}^{\frac{n}{2}-1} \bar{\theta}_i^\odd 
=(q^{n}+q^{-n})\frac{q^{n}-q^{-n}}{q^4-q^{-4}}.
$$
The trace of $C$ on $V$ is equal to 
$\sum_{i=0}^{\frac{n}{4}-1}\bar{a}_i^\odd+\bar{b}_{\frac{n}{4}-1}^\odd$.  
By \eqref{bo} and Lemmas \ref{lem:parameterLn_odd}(iv) and \ref{lem:traceLn_odd}(ii), this sum is equal to 
$$
\frac{1}{2}
\sum\limits_{i=0}^{\frac{n}{2}-1} \bar{a}_i^\odd 
+\bar{b}_{\frac{n}{4}-1}^\odd
=(q^{n+4}+q^{-n-4})\frac{q^{n}-q^{-n}}{q^4-q^{-4}}.
$$
According to Lemma \ref{thm:Vn(abc)_criterion}, the irreducible $\triangle_{q^4}$-module $V$ is isomorphic to $V_{\frac{n}{4}-1}(q^{n},q^{n},q^{n+4})$.

The actions of $A,B,C$ on $W$ are as follows:
\begin{align*}
A w_i&=
q^{2n-8i}\bar{c}_i^{\odd} w_{i-1}
+\bar{a}_i^\odd w_i
+q^{8i-2n+8}\bar{b}_i^\odd w_{i+1}
\qquad  (0\leq i\leq  \tfrac{n}{4}-2),
\\
A w_{\frac{n}{4}-1} &=
q^8 \bar{c}_{\frac{n}{4}-1}^{\odd} w_{\frac{n}{4}-2}
+(\bar{a}_{\frac{n}{4}-1}^\odd+\bar{b}_{\frac{n}{4}-1}^\odd) w_{\frac{n}{4}-1},
\\
B w_i&=\bar{\theta}_i^\odd w_i \qquad  (0\leq i\leq  \tfrac{n}{4}-1),
\\
C w_i&=
-\bar{c}_i^{\odd} w_{i-1}
+\bar{a}_i^\odd w_i
-\bar{b}_i^\odd w_{i+1}
\qquad   (0\leq i\leq \tfrac{n}{4}-2),
\\
C w_{\frac{n}{4}-1} &=
-\bar{c}_{\frac{n}{4}-1}^{\odd} w_{\frac{n}{4}-2}
+(\bar{a}_{\frac{n}{4}-1}^\odd-\bar{b}_{\frac{n}{4}-1}^\odd) w_{\frac{n}{4}-1},
\end{align*}
where $w_{-1}=0$. With respect to the ordered basis $\{w_i\}_{i=0}^{\frac{n}{4}-2}$, the matrix representing $B$ is diagonal, while those representing $A$ and $C$ are tridiagonal. 
The irreducibility of $W$ then follows from Lemma \ref{lem:parameterLn_odd}(ii) and (iii). 
The trace of $A$ on $W$ is equal to 
$\sum_{i=0}^{\frac{n}{4}-1}\bar{a}_i^\odd+\bar{b}_{\frac{n}{4}-1}^\odd$.  
By \eqref{bo} and Lemmas \ref{lem:parameterLn_odd}(iv) and \ref{lem:traceLn_odd}(ii), this sum is equal to 
$$
\frac{1}{2}
\sum\limits_{i=0}^{\frac{n}{2}-1} \bar{a}_i^\odd 
+\bar{b}_{\frac{n}{4}-1}^\odd
=(q^{n+4}+q^{-n-4})\frac{q^{n}-q^{-n}}{q^4-q^{-4}}.
$$
The trace of $B$ on $W$ is equal to 
$\sum_{i=0}^{\frac{n}{4}-1}\bar{\theta}_i^\odd$. 
By Lemmas \ref{lem:parameterLn_odd}(iv) and  \ref{lem:traceLn_odd}(i), this sum evaluates to 
$$
\frac{1}{2}
\sum\limits_{i=0}^{\frac{n}{2}-1} \bar{\theta}_i^\odd 
=(q^{n}+q^{-n})\frac{q^{n}-q^{-n}}{q^4-q^{-4}}.
$$
The trace of $C$ on $W$ is equal to 
$\sum_{i=0}^{\frac{n}{4}-1}\bar{a}_i^\odd-\bar{b}_{\frac{n}{4}-1}^\odd$.  
By \eqref{bo} and Lemmas \ref{lem:parameterLn_odd}(iv) and  \ref{lem:traceLn_odd}(ii), it is equal to 
$$
\frac{1}{2}\sum\limits_{i=0}^{\frac{n}{2}-1} \bar{a}_i^\odd 
-\bar{b}_{\frac{n}{4}-1}^\odd
=(q^{n}+q^{-n})\frac{q^{n}-q^{-n}}{q^4-q^{-4}}.
$$
By Lemma \ref{thm:Vn(abc)_criterion},  the irreducible $\triangle_{q^4}$-module $W$ is isomorphic to $V_{\frac{n}{4}-1}(q^{n+4},q^{n},q^{n})$.
\end{proof}

Theorems \ref{thm:decLn_even} and \ref{thm:decLn_odd} reveal a remarkable structural rigidity: the decomposition of finite-dimensional irreducible $\Uq$-modules of classical type into $\triangle _{q^4}$-modules remains perfectly consistent with the classical branching rules of $\U$ over $\Re$ established in Theorems \ref{thm:dec_Lne} and \ref{thm:dec_Lno}. This correspondence confirms that Theorem \ref{thm:UAW->Uq(so3)} is indeed the canonical $q$-analogue of Theorem \ref{thm:R->U(so3)}. Furthermore, the preservation of these decomposition patterns under $q$-deformation underscores the intrinsic nature of the cyclic symmetry shared by the universal Askey--Wilson algebra and the nonstandard quantum algebra $\Uq$. Having fully characterized the classical type, we now proceed to investigate the branching rules for the irreducible $\Uq$-modules of non-classical type.

\subsection{Decomposition of non-classical $\Uq$-modules over $\triangle_{q^4}$}\label{sec:nonclassical}

Let $n\in \N$ and $\e,\e'\in \{\pm 1\}$ be given. According to \cite{qso3:1998,Hav2001} there exists a unique $(n+1)$-dimensional irreducible $\Uq$-module $R_n(\e,\e')$ up to isomorphism satisfying the following conditions:
\begin{enumerate}

\item[$\bullet$] There exists a basis $\{u_i\}_{i=0}^n$ for $R_n(\e,\e')$ satisfying 
\begin{align*}
(q^2-q^{-2})  I_1 u_0 &=
-\imi \e' c_0 u_0
+ q^2 b_0 u_1,
\\
(q^2-q^{-2})  I_1 u_i &=
 q^{-2i} c_i u_{i-1}
+ q^{2i+2}b_i u_{i+1}
\qquad (1\leq i\leq n),
\\
\e(q^2-q^{-2}) I_2 u_i &=
 \theta_i u_i
\qquad (0\leq i\leq n),
\\
\e(q^2-q^{-2})  I_3 u_0 &=
-\imi \e' c_0 u_0
+b_0 u_1,
\\
\e (q^2-q^{-2})  I_3 u_i &=
 c_i u_{i-1}
+b_i u_{i+1}
\qquad (1\leq i\leq n),
\end{align*}
where $u_{n+1}$ is interpreted as the zero vector of $R_n(\e,\e')$ and 
\begin{align*}
\theta_i &=q^{2i+1}+q^{-2i-1}
\qquad (0\leq i\leq n),
\\
b_i&=
\imi \frac{q^{2n-2i}-q^{2i-2n}}{q^{2i+1}-q^{-2i-1}}
\qquad (0\leq i\leq n),
\\
c_i&=
\imi \frac{q^{2n+2i+2}-q^{-2n-2i-2}}{q^{2i+1}-q^{-2i-1}}
\qquad (0\leq i\leq n).
\end{align*}

\item[$\bullet$] The Casimir element $\Lambda$ acts on $R_n(\e,\e')$ as scalar multiplication by 
$$
\frac{(q^{2n+1}+q^{-2n-1})(q^{2n+3}+q^{-2n-3})}{(q^2-q^{-2})^2}.
$$
\end{enumerate}

The $\Uq$-module $R_n(\e,\e')$  is said to be of {\it non-classical} type.

\begin{lem}
\label{lem:UAWonRn}
For any $n\in \N$ and any $\e,\e'\in \{\pm 1\}$, the actions of $A,B,C$ on $R_n(\e,\e')$ are as follows:
\begin{align*}
A u_0 &=
\bar{a}_0 u_0
+\imi \e' q^2 \bar{c}_0 u_1
+q^6 \bar{b}_0 u_2,
\\
A u_1 &=
-\imi \e' q^{-2} \bar{c}_1 u_0
+ \bar{a}_1 u_1
+q^{10} \bar{b}_1 u_3
\qquad 
(n\geq 1),
\\
A u_i &=
q^{2-4i} \bar{c}_i u_{i-2}
+\bar{a}_i u_i
+q^{4i+6}\bar{b}_i u_{i+2}
\qquad (2\leq i\leq n),
\\
B u_i &=
-\bar{\theta}_i u_i 
\qquad (0\leq i\leq n),
\\
C u_0 &=
\bar{a}_0 u_0
+\imi \e' \bar{c}_0 u_1
+\bar{b}_0 u_2,
\\
C u_1 &=
-\imi \e' \bar{c}_1 u_0
+ \bar{a}_1 u_1
+\bar{b}_1 u_3
\qquad 
(n\geq 1),
\\
C u_i &=
\bar{c}_i u_{i-2}
+\bar{a}_i u_i
+\bar{b}_i u_{i+2}
\qquad (2\leq i\leq n),
\end{align*}
where $u_{n+1}$ and $u_{n+2}$ are interpreted as the zero vector of $R_n(\e,\e')$, and the coefficients are given by 
\begin{align}
\bar{\theta}_i &=q^{4i+2}+q^{-4i-2}
\qquad (0\leq i\leq n),
\label{thetaR}
\\
\bar{a}_i &=
\frac{(q^2+q^{-2})(q^{2n+1}-q^{-2n-1})(q^{2n+3}-q^{-2n-3})}{(q^{2i+3}-q^{-2i-3})(q^{2i-1}-q^{1-2i})}
\qquad (0\leq i\leq n),
\label{aR}
\\
\bar{b}_i &=\frac{(q^{2n-2i-2}-q^{2i-2n+2})(q^{2n-2i}-q^{2i-2n})}
{(q^{2i+1}-q^{-2i-1})(q^{2i+3}-q^{-2i-3})}
\qquad (0\leq i\leq n),
\label{bR}
\\
\bar{c}_i &=\frac{(q^{2n+2i}-q^{-2n-2i})(q^{2n+2i+2}-q^{-2n-2i-2})}
{(q^{2i-1}-q^{1-2i})(q^{2i+1}-q^{-2i-1})}
\qquad (0\leq i\leq n).
\label{cR}
\end{align}
Moreover, the central elements $\alpha,\beta,\gamma$ act on $R_n(\e,\e')$ as scalar multiplication by 
$$
-(q^2+q^{-2})(q^{2n+1}-q^{-2n-1})(q^{2n+3}-q^{-2n-3}).
$$
\end{lem}
\begin{proof}
The formulas are obtained by applying the algebra homomorphism from Theorem \ref{thm:UAW->Uq(so3)} to the $\Uq$-module $R_n(\e,\e')$. 
\end{proof}


\begin{lem}
\label{lem:parameterRn}
With the notation of Lemma \ref{lem:UAWonRn}, the following statements hold:
\begin{enumerate}
\item The scalars $\{\bar{\theta}_i \}_{i=0}^{n}$ are mutually distinct.

\item If $n\geq 2$ then the scalars $\{\bar{b}_i\}_{i=0}^{n-2}$  are all nonzero.

\item If $n\geq 1$ then the scalars  $\{\bar{c}_i\}_{i=0}^{n}$ are all nonzero.

\item For each $i=0,1,\ldots,n$ the scalar $\bar{a}_i$ is equal to 
$$
\frac{(q^{2n+1}-q^{-2n-1})(q^{2n+3}-q^{-2n-3})}{2(q^2-q^{-2})}
\left(
\frac{q^{2i-1}+q^{1-2i}}{q^{2i-1}-q^{1-2i}}
-
\frac{q^{2i+3}+q^{-2i-3}}{q^{2i+3}-q^{-2i-3}}
\right).
$$

\end{enumerate}
\end{lem}
\begin{proof}
(i): Since $q$ is not a root of unity, it follows from \eqref{thetaR} that for any distinct $i, j \in \{0, 1, \ldots, n\}$, the equality $\bar{\theta}_i=\bar{\theta }_j$ holds if and only if  $i+j=-1$, which is impossible. Thus (i) follows.

(ii), (iii): Under the assumption that $q$ is not a root of unity, the numerators of \eqref{bR} and \eqref{cR} do not vanish for the given indices.

(iv): This follows by applying the identity
$$
\frac{2(q^{a-b}-q^{b-a})}{(q^a-q^{-a})(q^b-q^{-b})}
=
\frac{q^b+q^{-b}}{q^b-q^{-b}}
-\frac{q^a+q^{-a}}{q^a-q^{-a}}
$$
to \eqref{aR} with $a=2i+3$ and $b=2i-1$.
\end{proof}

\begin{lem}
\label{lem:traceRn}
For any $n\in \N$ and any $\e,\e'\in \{\pm 1\}$, the trace of each of $A,B,C$ on $R_n(\e,\e')$ is equal to 
$$
-\frac{q^{4n+4}-q^{-4n-4}}{q^2-q^{-2}}.
$$
\end{lem}
\begin{proof}
In this proof the parameters $\{\bar{\theta}_i\}_{i=0}^n$ and $\{\bar{a}_i\}_{i=0}^n$ are given in \eqref{thetaR} and \eqref{aR}. 
The trace of $B$ on $R_n(\e,\e')$ is equal to $-\sum_{i=0}^n \bar{\theta}_i$. By the summation formula for geometric series, the sum evaluates to $-1$ multiplied by 
$$
\sum_{i=0}^n q^{4i+2}+\sum_{i=0}^n q^{-4i-2}
=\frac{q^{4n+4}-q^{-4n-4}}{q^2-q^{-2}}.
$$
The trace of $A$ and $C$ are both equal to $\sum_{i=0}^n \bar{a}_i$. Applying Lemma \ref{lem:parameterRn}(iv), this sum is equal to $-\frac{(q^{2n+1}-q^{-2n-1})(q^{2n+3}-q^{-2n-3})}{2(q^2-q^{-2})}$ multiplied by a telescoping sum with staggered cancellation:
\begin{align*}
\sum_{i=0}^n
\frac{q^{2i+3}+q^{-2i-3}}{q^{2i+3}-q^{-2i-3}}
-
\frac{q^{2i-1}+q^{1-2i}}{q^{2i-1}-q^{1-2i}}
&=
\frac{q^{2n+3}+q^{-2n-3}}{q^{2n+3}-q^{-2n-3}}
+\frac{q^{2n+1}+q^{-2n-1}}{q^{2n+1}-q^{-2n-1}}
\\
&=
\frac{2(q^{4n+4}-q^{-4n-4})}{(q^{2n+3}-q^{-2n-3})(q^{2n+1}-q^{-2n-1})}.
\end{align*}
Simplifying the resulting expression yields the value stated in the lemma.
\end{proof}

\begin{thm}
For any $n\in \N$ and any $\e,\e'\in \{\pm 1\}$, the $\triangle_{q^4}$-module $R_n(\e,\e')$ is isomorphic to the irreducible $\triangle_{q^4}$-module $V_n(-q^2,-q^2,-q^2)$.
\end{thm}
\begin{proof}
Let $\{u_i\}_{i=0}^n$  be the basis for $R_n(\e,\e')$ given in Lemma \ref{lem:UAWonRn}. We give an ordered basis for $R_n(\e,\e')$ as follows:
\begin{align*}
&u_n, u_{n-2}, \ldots, u_0,u_1,u_3,\ldots, u_{n-1} \qquad \hbox{if $n$ is even},
\\
&u_{n-1}, u_{n-3},\ldots, u_1, u_0, u_2,\ldots, u_n \qquad \hbox{if $n$ is odd}.
\end{align*}
With respect to this ordered basis, the matrix representing $B$ is diagonal, while those representing $A$ and $C$ are tridiagonal. The irreducibility of $R_n(\e,\e')$ then follows from Lemma \ref{lem:parameterRn}(i)--(iii).

Finally, by Lemma \ref{lem:traceRn} the traces of $A,B,C$ on $R_n(\e,\e')$ are equal to 
$$
-(q^2+q^{-2})\frac{q^{4n+4}-q^{-4n-4}}{q^4-q^{-4}}.
$$
According to Lemma \ref{thm:Vn(abc)_criterion} the $\triangle_{q^4}$-module $R_n(\e,\e')$ is isomorphic to $V_n(-q^2,-q^2,-q^2)$. 
\end{proof}

\section{Realizations within anticommutator spin algebras via skew group rings}\label{s:anticommutator}

The Lie algebra $\mathfrak{so}_3$ is the complexification of the standard angular momentum algebra.
The {\it anticommutator spin algebra} $\A$ \cite{ACSA2003} is a fermionic analogue of the standard angular momentum algebra, obtained by replacing the commutation relations \eqref{so3-1}--\eqref{so3-3} with the corresponding anticommutation relations. 
Specifically, the anticommutator spin algebra $\A$ is defined as the algebra over $\C$ generated by $J_1, J_2, J_3$ subject to the relations
\begin{align}
\{J_1,J_2\}=J_3,
\label{ACSA-1}
\\
\{J_2,J_3\}=J_1,
\label{ACSA-2}
\\
\{J_3,J_1\}=J_2.
\label{ACSA-3}
\end{align}

To further elucidate the connection between $\mathfrak{so}_3$ and $\A$, 
we pass to the setting of skew group rings. 
By \eqref{so3-1}--\eqref{so3-3}, there is a unique Lie algebra involution $\rho:\mathfrak{so}_3\to \mathfrak{so}_3$ that maps
\begin{eqnarray}\label{rho}
I_1&\mapsto & I_1,
\qquad
I_2\;\;\mapsto\;\; -I_2,
\qquad 
I_3\;\;\mapsto\;\; -I_3.
\end{eqnarray}
This involution extends to an algebra involution of $U(\mathfrak{so}_3)$, denoted again by $\rho$. 
Let $\U \rtimes \langle \rho \rangle$ denote the corresponding skew group ring of $\langle \rho \rangle$ over $U(\mathfrak{so}_3)$; the structure was previously investigated in an alternative form in \cite{odd:2026}.
Analogously, based on \eqref{ACSA-1}--\eqref{ACSA-3}, there exists a unique algebra involution $\varrho:\mathcal A\to \mathcal A$ that maps
\begin{eqnarray}\label{varrho}
J_1 &\mapsto &  J_1,
\qquad 
J_2 \;\;\mapsto\;\;  -J_2,
\qquad 
J_3 \;\;\mapsto\;\;  -J_3.
\end{eqnarray}
We denote the resulting skew group ring by $\A\rtimes \langle \varrho \rangle$.

\begin{thm}
\label{thm:AZ/2Z->UZ/2Z}
\begin{enumerate}
\item There exists a unique algebra homomorphism $\A\rtimes \langle \varrho \rangle \to \U \rtimes \langle \rho \rangle$ that maps 
\begin{gather*}
J_1 \;\;\mapsto\;\;  
\imi I_1\rho,
\qquad
J_2 \;\;\mapsto\;\;  
\imi I_2,
\qquad 
J_3 \;\;\mapsto\;\;  
I_3\rho,
\\
\varrho \;\;\mapsto\;\;  \rho.
\end{gather*}

\item There exists a unique algebra homomorphism $U(\mathfrak{so}_3) \rtimes \langle \rho \rangle\to \A\rtimes \langle \varrho \rangle$ that maps 
\begin{gather*}
I_1 \;\;\mapsto\;\; -\imi J_1 \varrho,
\qquad
I_2 \;\;\mapsto\;\; -\imi J_2,
\qquad 
I_3 \;\;\mapsto\;\; J_3\varrho,
\\
\rho \;\;\mapsto\;\; \varrho.
\end{gather*}

\item The algebra homomorphisms given in {\rm (i)} and {\rm (ii)} are inverses to each other.
\end{enumerate}
\end{thm}
\begin{proof}
It is routine to verify (i) and (ii) by using \eqref{so3-1}--\eqref{so3-3}, \eqref{ACSA-1}--\eqref{ACSA-3} as well as \eqref{rho} and \eqref{varrho}. Statement (iii) is immediate from (i) and (ii).
\end{proof}

Crucially, Theorem \ref{thm:AZ/2Z->UZ/2Z}(ii) facilitates the construction of an anticommutator analogue of Theorem \ref{thm:R->U(so3)}, where $U(\mathfrak{so}_3)$ is replaced by $\A$.
This realization does not follow directly from previous results such as \cite[Theorem 4.1]{Huang:R<BI}.

\begin{thm}\label{thm:R->A}
There exists a unique algebra homomorphism
$\Re \to \mathcal{A}$
that sends
\begin{eqnarray}
A &\mapsto & 
\frac{J_1^2-1}{4},
\label{R->A-1}
\\
B &\mapsto & 
\frac{J_2^2-1}{4},
\label{R->A-2}
\\
C &\mapsto & 
\frac{J_3^2-1}{4}.
\label{R->A-3}
\end{eqnarray}
Moreover, the homomorphism maps each of $\alpha,\beta,\gamma$ to zero.
\end{thm}
\begin{proof}
The algebra homomorphism $\Re \to \A$ can be obtained by composing the algebra homomorphism $\Re \to \U$ with the inclusion map $\U \overset{\scriptscriptstyle\subseteq}{\hookrightarrow} \U\rtimes \langle \rho \rangle$, followed by the algebra homomorphism $\U\rtimes \langle \rho \rangle \to \A\rtimes \langle \varrho \rangle$ from Theorem~\ref{thm:AZ/2Z->UZ/2Z}(ii).
\end{proof}

The relationships among the universal Racah algebra $\Re$, the Lie algebra $\mathfrak{so}_3$, and the anticommutator spin algebra $\mathcal A$ are summarized in the commutative diagram below. The arrows  
$\Re\to \U$,
$\A\rtimes \langle \varrho \rangle\to 
\U\rtimes \langle \rho \rangle$,
$\U\rtimes \langle \rho \rangle\to \A\rtimes \langle \varrho \rangle$
and 
$\Re\to \mathcal A$ denote the algebra homomorphisms established in Theorems {\rm \ref{thm:R->U(so3)}}, {\rm \ref{thm:AZ/2Z->UZ/2Z}} and {\rm \ref{thm:R->A}}.

\begin{figure}[H]
\centering
\[
\begin{tikzcd}[row sep=1.85em, column sep=4.2em]
& \U \arrow[r, hookrightarrow, "\subseteq"] 
  & \U\rtimes \langle \rho \rangle
    \arrow[dd, shift left=0.7ex]
\\
\Re \arrow[ru, bend left=15] \arrow[rd, bend right=15] 
  & & 
\\
& \A\arrow[r, hookrightarrow, "\subseteq"] 
  & \A\rtimes \langle \varrho \rangle
    \arrow[uu, shift left=0.7ex]
\end{tikzcd}
\]
\caption{Relations among the universal Racah algebra $\Re$, the Lie algebra $\mathfrak{so}_3$ and the anticommutator spin algebra $\A$}
\end{figure}

We now extend this correspondence to the quantum setting by considering $\A_q$, a $q$-deformation of the anticommutator spin algebra defined as the algebra over $\C$ generated by $J_1, J_2, J_3$ subject to the relations
\begin{align}
\{J_1, J_2\}_q
&=J_3,
\label{Aq-1}
\\
\{J_2, J_3\}_q
&=J_1,
\label{Aq-2}
\\
\{J_3, J_1\}_q
&=J_2.
\label{Aq-3}
\end{align}
Similarly, we extend $\Uq$ and $\A_q$ by adjoining a $\Z/2\Z$-symmetry via skew group rings. 
By \eqref{Uq-1}--\eqref{Uq-3} there exists a unique algebra involution 
$\rho:\Uq\to \Uq$ given by 
\begin{eqnarray*} 
I_1 &\mapsto& I_1, 
\qquad 
I_2 \mapsto -I_2, 
\qquad 
I_3 \mapsto -I_3.
\end{eqnarray*}
The associated skew group ring of $\langle \rho \rangle$ over $\Uq$ will be denoted by $\Uq\rtimes\langle \rho \rangle$. 
By \eqref{Aq-1}--\eqref{Aq-3} there exists a unique algebra involution 
$\varrho:\A_q\to\A_q$ given by
\begin{eqnarray*} 
J_1 &\mapsto& J_1, 
\qquad 
J_2 \mapsto -J_2, 
\qquad 
J_3 \mapsto -J_3.
\end{eqnarray*}
The corresponding skew group ring of $\langle \varrho \rangle$ over $\A_q$ will be denoted by $\A_q\rtimes\langle \varrho \rangle$.  The correspondence established in Theorem \ref{thm:AZ/2Z->UZ/2Z} admits a natural $q$-analogue as follows.

\begin{thm}
\label{thm:AqZ/2Z->UqZ/2Z}
\begin{enumerate}
\item There exists a unique algebra homomorphism $\A_q\rtimes\langle \varrho \rangle
\to \Uq\rtimes\langle \rho \rangle$ that maps 
\begin{gather*}
J_1 \;\;\mapsto\;\;  
\imi I_1\rho,
\qquad
J_2 \;\;\mapsto\;\;  
\imi I_2,
\qquad 
J_3 \;\;\mapsto\;\;  
I_3\rho,
\\
\varrho \;\;\mapsto\;\;  \rho.
\end{gather*}

\item There exists a unique algebra homomorphism $\Uq\rtimes\langle \rho \rangle\to \A_q\rtimes\langle \varrho \rangle$ that maps 
\begin{gather*}
I_1 \;\;\mapsto\;\; -\imi J_1 \varrho,
\qquad
I_2 \;\;\mapsto\;\; -\imi J_2,
\qquad 
I_3 \;\;\mapsto\;\; J_3\varrho,
\\
\rho \;\;\mapsto\;\; \varrho.
\end{gather*}

\item The algebra homomorphisms given in {\rm (i)} and {\rm (ii)} are inverses to each other.
\end{enumerate}
\end{thm}

As a consequence, the isomorphism $\Uq\rtimes\langle \rho \rangle\to \A_q\rtimes\langle \varrho \rangle$ maps the Casimir element $\Lambda$ of $\Uq$ to the negative of the Casimir element $\Pi$ of $\A_q$, where
$$
\Pi = q^2 J_1^2 +q^{-2} J_2^2+q^2 J_3^2 
     - q(q^2 - q^{-2}) J_1 J_2 J_3.
$$

\begin{thm}
\label{thm:UAW->Aq}
There exists a unique algebra homomorphism $\triangle_{q^4}\to \A_q$ that sends 
\begin{eqnarray*}
A &\mapsto & 2+(q^2-q^{-2})^2 J_1^2,
\\
B &\mapsto & 2+(q^2-q^{-2})^2 J_2^2,
\\
C &\mapsto & 2+(q^2-q^{-2})^2 J_3^2.
\end{eqnarray*}
Moreover, the homomorphism maps each of $\alpha,\beta,\gamma$ to 
$$
2(q^2+q^{-2})^2+(q^2-q^{-2})(q^4-q^{-4}) \Pi.
$$
\end{thm}
\begin{proof}
Similar to the proof of Theorem \ref{thm:R->A}.
\end{proof}

\section{Concluding remarks: Higher-rank generalizations}
\label{s:remark}

The algebraic correspondence established in this paper for $\Re $ and $\triangle_q$ provides a robust foundation for exploring the higher-rank generalizations of Askey--Wilson structures. While the rank $1$ case, as treated here, features an inherent cyclic symmetry and a direct connection to the nonstandard quantum algebra $\Uq$, the path toward higher rank $n\ge 2$ typically diverges into two distinct directions.

On one hand, the higher-rank universal Askey--Wilson algebra $\triangle_{q}(n)$ can be systematically formulated via the universal $R$-matrix and the associated braiding symmetries \cite{qBI2018,HR2017,qAW&BI:2019,qAW&Skein:2025}. However, such formulations rely on intricate recursive constructions via nested coproducts and braiding operators, which often entail a prohibitive increase in algebraic complexity. Such frameworks, while theoretically comprehensive, frequently mask the structural elegance and the intrinsic rotational symmetry that remain transparent in the rank 1 setting. On the other hand, the nonstandard quantum algebras $U_{q}^{\prime }(\mathfrak{so}_{n})$ maintain a high degree of structural elegance, yet their precise role in elucidating the bispectrality of multivariable $q$-orthogonal polynomials remains an open problem of considerable depth.

By establishing the rigorous branching rules and the bosonic-fermionic correspondence for rank $1$, we have demonstrated that the ``rigidity'' of the Askey--Wilson patterns persists across both classical and quantum settings. Whether this rigidity can be preserved in higher rank without resorting to the full machinery of the universal $R$-matrix remains an intriguing open question. We hope that the framework developed in this study serves as a cornerstone toward a more unified theory of higher-rank Askey--Wilson symmetry.

\subsection*{Funding statement}
The research was supported by the National Science and Technology Council of Taiwan under the grant NSTC 114-2115-M-008-011.

\subsection*{Conflict-of-interest statement} 
The author has no relevant financial or non-financial interests to disclose.

\subsection*{Data availability statement}
No datasets were generated or analyzed during the current study.

\bibliographystyle{amsplain}
\bibliography{MP}

\end{document}